\documentclass[final,reqno]{siamart190516}

\usepackage{braket,amsfonts}
\usepackage{siunitx}

\usepackage{array}

\usepackage[caption=false]{subfig}

\usepackage{pgfplots}
\pgfplotsset{compat=1.14}

\newsiamthm{claim}{Claim}
\newsiamremark{rem}{Remark}
\newsiamremark{expl}{Example}
\newsiamremark{hypothesis}{Hypothesis}
\crefname{hypothesis}{Hypothesis}{Hypotheses}

\usepackage{algorithmic}
\usepackage{algorithm}

\usepackage{graphicx}
\usepackage{epstopdf}
\usepackage{multicol}

\Crefname{ALC@unique}{Line}{Lines}

\usepackage{amsopn}

\numberwithin{theorem}{section}

\usepackage{xspace}
\usepackage{bold-extra}
\usepackage[most]{tcolorbox}

\newcommand{\R}{\mathbb{R}}

\colorlet{texcscolor}{blue!50!black}
\colorlet{texemcolor}{red!70!black}
\colorlet{texpreamble}{red!70!black}
\colorlet{codebackground}{black!25!white!25}

\usepackage{amsmath,amssymb, amsfonts}
\usepackage{setspace}
\usepackage[mathscr]{euscript}
\usepackage{color}
\usepackage{cite}
\usepackage{booktabs} 
\usepackage{enumerate}
\usepackage{tikz-cd}

\lstdefinestyle{siamlatex}{%
  style=tcblatex,
  texcsstyle=*\color{texcscolor},
  texcsstyle=[2]\color{texemcolor},
  keywordstyle=[2]\color{texemcolor},
  moretexcs={cref,Cref,maketitle,mathcal,text,headers,email,url},
}

\tcbset{%
  colframe=black!75!white!75,
  coltitle=white,
  colback=codebackground, % bottom/left side
  colbacklower=white, % top/right side
  fonttitle=\bfseries,
  arc=0pt,outer arc=0pt,
  top=1pt,bottom=1pt,left=1mm,right=1mm,middle=1mm,boxsep=1mm,
  leftrule=0.3mm,rightrule=0.3mm,toprule=0.3mm,bottomrule=0.3mm,
  listing options={style=siamlatex}
}

\newtcblisting[use counter=example]{example}[2][]{%
  title={Example~\thetcbcounter: #2},#1}

\newtcbinputlisting[use counter=example]{\examplefile}[3][]{%
  title={Example~\thetcbcounter: #2},listing file={#3},#1}

\DeclareTotalTCBox{\code}{ v O{} }
{ %fontupper=\ttfamily\color{texemcolor},
  fontupper=\ttfamily\color{black},
  nobeforeafter,
  tcbox raise base,
  colback=codebackground,colframe=white,
  top=0pt,bottom=0pt,left=0mm,right=0mm,
  leftrule=0pt,rightrule=0pt,toprule=0mm,bottomrule=0mm,
  boxsep=0.5mm,
  #2}{#1}

\patchcmd\newpage{\vfil}{}{}{}
\begin{tcbverbatimwrite}{tmp_\jobname_header.tex}
\title{Multiscale Global sensitivity analysis for stochastic chemical systems
  \thanks{Submitted to the editors \today.
\funding{This work was supported by the National Science Foundation under grant DMS-1745654.}}}

\author{M. Merritt\thanks{Department of Mathematics, North Carolina State
University, Raleigh, NC 27695-8205 (\email{mbmerrit@ncsu.edu})}  \and A.~Alexanderian\thanks{Department of Mathematics, North Carolina State University,
Raleigh, NC 27695-8205 (\email{alexanderian@ncsu.edu})}\and P.A.~Gremaud\thanks{Department of Mathematics, North Carolina State University, Raleigh, NC (\email{gremaud@ncsu.edu}).}
}

\headers{Multiscale GSA for stochastic chemical systems}
{M. MERRITT, A. ALEXANDERIAN,  AND P.A.~GREMAUD}
\end{tcbverbatimwrite}
\input{tmp_\jobname_header.tex}

\begin{document}
\maketitle

%% ------------------------------------------------------------------
%% ABSTRACT
%% ------------------------------------------------------------------
\begin{tcbverbatimwrite}{tmp_\jobname_abstract.tex}
\begin{abstract}
Sensitivity analysis is routinely performed on simplified surrogate models as
the cost of such analysis  on the original model may be prohibitive. Little is
known in general about the induced bias on the sensitivity results. Within the
framework of chemical kinetics, we provide a full justification of the above
approach in the case of variance based methods provided the surrogate model
results from the original one through the thermodynamic limit.  We also provide
illustrative numerical examples in context of a Michaelis--Menten system  and a
biochemical reaction network describing a genetic oscillator.

\end{abstract}


\end{tcbverbatimwrite}
\input{tmp_\jobname_abstract.tex}
%% ------------------------------------------------------------------
%% END HEADER
%% ------------------------------------------------------------------

\section{Introduction}\label{sec:intro}
Striking a balance between accuracy and cost is one of the core challenges of
scientific computing. A high fidelity, high cost model $g$ is thus often replaced in practice by a lower cost model $\tilde g$, of (usually) lower fidelity, to enable the analysis of the application under study. The techniques   to develop and construct surrogate models are many and range from approximation theory to physics \cite{surrogate-koziel}. The analysis of the original model $g$ is then replaced by the analysis of a surrogate $\tilde g$ with the implicit assumption that 
\begin{eqnarray}
\mbox{if } g \approx \tilde g \mbox{ then } \mathcal I(g) \approx \mathcal I(\tilde g), \label{general-q}
\end{eqnarray}
where $\mathcal I$ represents some operation on $g$. The extent to which
(\ref{general-q}) is satisfied clearly depends on $\mathcal I$ and on the
relationship between $g$ and $\tilde g$.  This paper is a first step toward the
justification of (\ref{general-q}) when $\mathcal I$ stands for the sensitivity
of the model to its input parameters. We restrict our attention to an important
family of physically based surrogates corresponding to $\tilde g$ being the
thermodynamic limit of $g$ and take chemical reaction networks as a motivating
application. Recent results about  approximation based---rather than physically
based---surrogates can be found in~\cite{qian}.

Consider thus the evolution of a system of chemically reacting molecules;
molecular dynamics simulation is the most faithful way of modeling such a
system. There, each individual molecule and corresponding species population
are tracked and chemical reactions are modeled as distinct events. Due to
quantum effects and since such systems are typically not isolated, molecular
populations are integer variables which evolve stochastically
\cite{gillespie07}.    In spite of  this,  chemical kinetics is often analyzed
using real---as opposed to integer---variables which evolve deterministically;
that this is the case is a testimony to the appeal of simplified low-cost
models. {\em Stochastic} chemical kinetics is  however necessary to the study
of many cellular systems in biology where the relatively small molecular
populations may preclude the use of simplified models obtained through the
thermodynamic limit, i.e., in the limit of large volumes  and may require a
stochastic rather than deterministic model.

Assume we have both a high cost stochastic model $g$ and a low cost
deterministic surrogate $\tilde g$ such that 
\begin{eqnarray}
q = g(\mathbf k,\omega), \quad \tilde q = \tilde g(\mathbf k) \quad \mbox{ and } q \approx \tilde q \mbox{ in some sense,} \label{qoi}
\end{eqnarray}
where the outcome $\omega$ corresponds to the intrinsic stochasticity of the
model $g$ and $q$ and $\tilde q$ are the respective quantities of interest
(QoIs); here $\mathbf k = (k_1, \dots, k_M)$ is a list of shared uncertain
parameters.  As shown below, the field of chemical kinetics falls under this
framework.

Global sensitivity analysis (GSA) aims to quantify the relative importance of
uncertain model parameters in determining the QoI
\cite{iooss,janon14,saltellibook}.  We analyze whether GSA can be performed on
the surrogate $\tilde g$ rather than $g$ and still yield information on the
original model $g$. In other words, we are asking when the diagram in
Figure~\ref{fig:diagram} is commutative.

\begin{figure}[ht]
\[\begin{tikzcd}
q = g(\mathbf k,\omega)  \arrow{r}{\mbox{ GSA }} \arrow[swap]{d}{\mbox{limiting process}} & \{\mathcal I_j(\omega)\}_{j=1}^M \arrow{d}{\mbox{limiting process}\\} \\
\tilde q = \tilde g(\mathbf k) \arrow{r}{\mbox{ GSA }} & \{\tilde {\mathcal I} _j\}_{j=1}^M
\end{tikzcd}
\]
\caption{Schematic representation of the question considered in this paper: for
what type of limiting process is the diagram commutative? The model $g$ is
expensive-to-evaluate and stochastic while the surrogate model $\tilde g$ is
deterministic and cheap.  We show that the diagram is commutative if the
limiting process is the thermodynamic limit.}
\label{fig:diagram}
\end{figure}

In Figure~\ref{fig:diagram}, $\mathcal I$ and $\tilde {\mathcal I} $ refer to
importance indices from some GSA method; presumably, when applied to stochastic
models, the GSA approach yields indices which themselves are random variables.
This is for instance the case for variance based methods and Sobol' indices
which we use in this paper, see \cite{hag} and  Section~\ref{sec:gsa}. For
chemical kinetics, the {\em limiting process}  in the above diagram is the thermodynamic limit, see Section~\ref{sec:models}. The above diagram does not in general commute; see \cite{hag} for simple analytical 
examples of non-commutativity when the 
 limiting process linking the stochastic model to its
surrogate is the expectation or some other $\omega$-moment.

\section{Chemical kinetics models}
\label{sec:models}
We consider chemical systems with $N$ reacting species.  We let $\mathbf X(t)$
be the state vector of a chemcial system, where $X_i(t)$, the $i$th component
of $\mathbf X(t)$,  corresponds to the number of molecules of $i$th species,
$i=1,\dots, N$, at time $t$.  

\subsection{The RTC representation} 
To guide our discussion, consider the simple case
of one reaction and three species $S_1$, $S_2$ and $S_3$
\begin{eqnarray}
S_1 + S_2 \to S_3, \label{baby}
\end{eqnarray}
where one molecule of $S_1$ and one molecule of $S_2$ combine to produce 
one molecule of $S_3$. The evolution of the state 
$\mathbf X(t) = \begin{bmatrix} X_1(t) & X_2(t) & X_3(t)\end{bmatrix}^\top$ 
takes the form
\begin{eqnarray}
\mathbf X(t) = \mathbf X(0) + \boldsymbol\nu R(t), \label{threespecies}
\end{eqnarray}
where $\boldsymbol\nu = \begin{bmatrix}-1 &  -1 &  1\end{bmatrix}^\top$ is the
stoichiometric vector of that reaction ($S_1$ and $S_2$ lose one molecule and
$S_3$ gains one) while $R(t)$ is the number of times the reaction takes place
between time 0 and $t$. It is intuitive, and has been justified on physical
ground \cite{gillespie07, higham}, that the probability of the reaction
occurring between time $t$ and $t + dt$ is proportional to $X_1(t)$, $X_2(t)$
and $dt$ which suggests the model~\cite{anderson,kurtz}
\begin{eqnarray}
R(t) = Y\left(\int_0^t c \, X_1(s) X_2(s)\, ds\right),  \label{firstp}
\end{eqnarray}
where $c$ a proportionality constant and $Y$ is a unit-rate Poisson process: $Y(0) = 0$, $Y$ has independent increments, and $Y(t+s) - Y(s)$ has a 
Poisson distribution with parameter $t$ for all $t,s \ge 0$, i.e.,
$\mathbb{P}\big(Y(t+s) - Y(s) =n\big) = e^{-t} t^n / n!$. 

More generally, the evolution of a system with $N$ species and $M$ reactions is
governed by the propensity functions $a_j$, $j=1,\dots, M$, where  $a_j(\mathbf
X(t))\, dt$ represents the probability that the $j$th reaction occurs during
the time interval $[t, t+dt)$. For instance, in the case of (\ref{baby}), the
propensity function is $a(\mathbf X(t)) = c\, X_1(t)X_2(t)$. The resulting
evolution equation, often referred to as the {\em random time change
representation} (RTC) \cite{anderson,ahlw,kurtz,AndersonHigham12},  is then 
\begin{eqnarray}
\mathbf X(t) = \mathbf X(0) + \sum_{j=1}^M \boldsymbol\nu_j Y_j\left(\int_0^t a_j(\mathbf X(s))\, ds \right), \label{rtc}
\end{eqnarray}
where $\boldsymbol\nu_j$ is the stoichiometric vector of the $j$th reaction and the $Y_j$'s are independent unit-rate Poisson processes. 
The Law of Mass Action \cite{anderson} leads to  the propensity functions for
the main three types of reactions: 
\begin{align}
   &\phantom{S_m +\,\, } S_m  &\to& \quad\mbox{something } 
   &\Rightarrow& \quad a_j(\mathbf X(t)) = c_j  X_m(t), 
   \label{first} \\
   &S_m + S_n &\to& \quad\mbox{something } 
   &\Rightarrow& \quad a_j(\mathbf X(t)) = c_j  X_m(t) X_n(t) \quad\text{if } m\ne n, 
   \label{second}\\
   &S_m + S_m &\to& \quad\mbox{something } 
   &\Rightarrow& \quad a_j(\mathbf X(t)) = c_j  \frac 12 X_m(t) (X_m(t)-1).  
   \label{dimerization}
\end{align}
The reactions~\eqref{first},~\eqref{second}, and~\eqref{dimerization} 
are known as first order, second order, and dimerization reactions, respectively.
The form of the propensity functions for other common 
reaction types can be found, for example, in~\cite{gillespie76}.

\subsection{The thermodynamic limit}
In our analysis, we consider the limiting behavior of chemical systems as the system size approaches infinity. For example, as the system size increases, the likelihood of a particular reaction to fire may change, in the event that certain molecules must interact. To this end, we aim to update the propensity functions by introducing a system
size parameter $V$ given by the product of the system volume and the Avogadro 
number $n_A$.
As is common in the study of chemical systems, 
we write the stoichiometric vectors as follows: 
\[
   \boldsymbol\nu_j = \boldsymbol\nu_j' - \boldsymbol\nu_j'', \quad j = 1, \ldots, M,
\]
where the entries of $\boldsymbol\nu_j'$ and $\boldsymbol\nu_j''$ are 
the number of molecules of system species that are created and consumed
in the $j$th reaction, respectively.
Following the notation of~\cite{WangRathinam16}, we define
the $V$-dependent propensity functions as follows:
\[
a_j^V(\mathbf x) = \frac{k_j}{V^{\|\nu_j''\|^-1}} \prod_{i=1}^N \binom{x_i}{\nu_{ij}''},
\quad j = 1, \ldots, M,
\]
where $k_j$'s are reaction rate constants.
The $V$-dependent system trajectory is described by the RTC representation,
\begin{equation}\label{rtcV}
\mathbf X^V(t) = V \mathbf x_0 + 
    \sum_{j=1}^M \boldsymbol\nu_j
       Y_j\left(\int_0^t a_j^V(\mathbf X^V(s))\, ds \right). 
\end{equation}
Here we have let $\mathbf X^V(0) = V \mathbf x_0$ where $\mathbf x_0 \in
\mathbb{R}^N_{\geq 0}$ is a fixed vector. Throughout we will work with a
sequence of $V$ values such that $V \mathbf{x}_0$ is in $\mathbb{Z}^N_{\geq 0}$. 
Ensuring existence of such a sequence requires some assumptions on
$\mathbf{x}_0$ and the nominal (initial) system volume.  Specifically, in our study of
limiting behavior of systems, we may assume that the system's nominal volume
$\mathcal{V}_\text{nom}$ and $\mathbf{x}_0$ are such that $V_\text{nom}
\mathbf{x}_0 = \mathcal{V}_\text{nom}n_A \mathbf{x}_0$ is a vector in
$\mathbb{Z}^N_{\geq 0}$.  We then consider a sequence of system sizes given by
$V_m = m V_\text{nom}$, $m = 1, 2, \ldots$. 

Notice that the RTC formulations~\eqref{rtcV} is a restatement of~\eqref{rtc},
except with the dependence on system size made precise. For instance, 
considering the system at its nominal volume  
$\mathcal{V}_\text{nom}$, $\mathbf{X}(0)$ in~\eqref{rtc} 
is given by 
\[
\mathbf X(0) = \mathbf X^{V_\text{nom}}(0) =  
V_\text{nom} \mathbf{x}_0 = \mathcal{V}_\text{nom}n_A \mathbf{x}_0.
\]

Next, we define the limiting 
propensity functions~\cite{WangRathinam16},
\[
    \bar{a}_j(\mathbf x) = \lim_{V \to \infty} a_j^V(V \mathbf x) / V, 
    \quad j = 1, \ldots, M. 
\]
For example, if the $j$th reaction is as in~\eqref{second},
\[
    a_j^V(\mathbf x) = \frac{k_j}{V} x_m x_n
    \quad \text{and} \quad
    \bar{a}_j(\mathbf x) = k_j x_m x_n 
\]
One the other hand, if the $j$th reaction is of the form~\eqref{dimerization}, 
\[
     a_j^V(\mathbf x) = \frac{k_j}{2V} x_m(x_m  - 1) 
    \quad \text{and} \quad
    \bar{a}_j(\mathbf x) = \frac12k_j x_m^2.
\]
To describe the thermodynamic limit, we consider the concentration-based
state vector $\mathbf Z^V(t) = \mathbf{X}^V / V$. In the limit as 
$V \to \infty$, $\mathbf Z^V(t)$ approaches, almost surely, to a deterministic
function $\mathbf{Z}(t)$ that is obtained by solving a system of ODEs
known as the system of reaction rate equations (RREs). The theoretical
result underpinning this is given in~\cite[Theorem 2.1 in Chapter 11]{kurtz}. 
Below, we follow the 
form of this result as presented in~\cite{WangRathinam16}.
We also point the reader to~\cite[Chapter 2]{ting}, 
for a detailed exposition of this result.

The concentration
vector $\mathbf Z^V$ follows the RTC representation~\cite{WangRathinam16},
\begin{equation} \label{rtc_conc}
\mathbf Z^V(t) = \mathbf{x}_0 + \sum_{j=1}^M \boldsymbol{\nu}_j
  V^{-1} Y_j\left( \int_0^t a_j^V(V \mathbf Z^V(s))ds \right).
\end{equation}
The corresponding system of RREs is described by
\begin{equation}\label{equ:RREs}
\begin{aligned}
\frac{d\mathbf Z}{dt} &= F(\mathbf Z(t)) \quad t \in [0, T],\\ 
\mathbf Z(0) &= \mathbf {x}_0,
\end{aligned}
\end{equation}
where $F(\mathbf z) =  \sum_{j=1}^M \, \boldsymbol\nu_j\bar{a}_j(\mathbf z)$
and $[0, T]$ is the maximal interval of existence of solution
for~\eqref{equ:RREs}.
The result given in~\cite[Theorem 2.1 in Chapter 11]{kurtz} 
(see also~\cite{WangRathinam16}),
which covers more general classes of Markov processes, states that
if for all compact $K \subset \R^N$
\begin{equation}\label{equ:conds}
\begin{aligned}
&\sum_{j=1}^M  \| \boldsymbol\nu_j \| \sup_{\mathbf z \in K} \bar a_j(\mathbf z) < \infty,
\quad \text{and} 
\\
&F \text{ is Lipschitz on } K,
\end{aligned}
\end{equation}
then 
\begin{equation}
\label{kurtz}
\lim _{V\to \infty} \sup _{s \le T} \| \mathbf Z^V(s) - \mathbf Z(s)\| = 0 \quad \mbox{almost surely}. 
\end{equation}
Therefore, we know that in the limit, as $V \to \infty$, the stochastic
solutions obtained from~\eqref{rtc_conc} will converge almost surely to the
solution of the ODE system~\eqref{equ:RREs}. Note also that 
both of the conditions in~\eqref{equ:conds} hold for the chemical systems under study,
because $\bar{a}_j$'s are polynomials.

\section{The Next Reaction Method}

Several algorithms have been developed for simulating the dynamics of a
stochastic chemical reaction network; these include 
Gillespie's stochastic simulation
algorithm (SSA) \cite{higham, gillespie07} as well as  the Next Reaction Method
(NRM) of Gibson and Bruck \cite{gibson00} and its variants
\cite{anderson2007modified,le2015variance,navarro2016global}.  The NRM
approach has a number of advantages over the SSA, see  \cite[Section
1]{anderson2007modified} and \cite[Section 3.B]{rathinam2010efficient}, among
others: (i) it is  cheaper to simulate than the SSA in terms of random 
numbers generated 
per iteration; and (ii) it has the ability to handle time-dependent propensity
functions and reactions that exhibit delays between initiation and completion.
The variant of the NRM that we  use below is developed by Anderson in
\cite{anderson2007modified}, where it is referred to as the \textit{modified
next reaction method}. 

\renewcommand{\algorithmicrequire}{\textbf{Input:}}
\renewcommand{\algorithmicensure}{\textbf{Output:}}
\begin{algorithm}[h]
\caption{Modified Next Reaction Method~\cite{anderson2007modified}.}\label{alg:NRM}
\begin{algorithmic}[1]

\REQUIRE Initial state $\mathbf X_0$, final simulation time $T$, stoichiometric 
matrix $\boldsymbol\nu$, and propensity functions,
$\{a_j(\cdot)\}_{j=1}^M$.
\ENSURE A realization of $\mathbf X(t, \omega)$.

\STATE \verb+% initialization %+ 
\FOR{$j = 1, \dots, M$}
\STATE Generate random number $r_j \sim U(0, 1)$
\STATE $\tau_j = 0, ~\tau_j^+ = -\ln (r_j)$ 
\ENDFOR
\STATE $t = 0, ~\textbf{X}(0) = \mathbf X_0$
\STATE \verb+% simulation loop %+
\WHILE{$t < T$}
\FOR{$j = 1, \dots, M$}
\STATE Evaluate $a_j(\textbf{X}(t))$ and $\Delta t_j = \frac{\tau_j^+ - \tau_j}{a_j(\textbf{X}(t))}$
      
\ENDFOR
\STATE Set $l = \underset{j}{\mathrm{argmin}}\{ \Delta t_j \}_{j=1}^M$
\STATE $\textbf{X}(t+\Delta t_l) \leftarrow \textbf{X}(t) + \boldsymbol{\nu}_l$ \hfill\COMMENT{Update state vector}
\STATE $t \leftarrow t + \Delta t_l$ \hfill\COMMENT{Update global time}
\FOR{$j = 1, \dots, M$}
\STATE $\tau_j \leftarrow \tau_j + a_j \Delta t_l$ \hfill\COMMENT{Update internal times of each reaction}
\ENDFOR
\STATE Generate random number $r_l \sim U(0,1)$

\STATE $\tau_l^+ \leftarrow \tau_l^+ - \ln (r_l)$ \hfill\COMMENT{Update next reaction time for reaction $l$}
\ENDWHILE \\

\end{algorithmic}
\end{algorithm}

Following \cite{gibson00}, we define an internal time $\tau_j$, for each reaction as 
\begin{equation} \label{eq:int_time}
\tau_j(t) = \int_0^t a_j(\textbf{X}(s)) \, ds, \qquad j=1,\dots, M.
\end{equation}
The NRM simulates RTC dynamics by treating each reaction as an independent stochastic process:
from \eqref{eq:int_time}, one can see that \eqref{rtc} is a linear combination of Poisson processes with different internal times  $\tau_j$, $j=1, \dots, M$. The approach is then to track the firing of each reaction in terms of these  internal times. Given the ``current" internal time $\tau_j$, $j=1,\dots, M$,  we denote by $\tau_j^+$ the internal time at which reaction $j$ fires next. At each iteration, the vectors 
$\begin{bmatrix} \tau_1 & \tau_2 & \cdots & \tau_M\end{bmatrix}^\top$ and
$\begin{bmatrix} 
\tau_1^+ & \tau_2^+\cdots & \tau_M^+\end{bmatrix}^\top$ 
store the current internal time and the next internal time for each reaction. 
Given these two vectors, 

one can determine
how much physical or global time will elapse before reaction $j$ 
fires again by considering
\[
\Delta t_j =  \frac{\tau_j^+ - \tau_j}{a_j(\textbf{X}(t))},
\quad 
j=1,\dots, M. 
\]
This is a direct consequence of \eqref{eq:int_time} and the assumption that
$a_j$ remains constant in the interval $[t, t + \Delta t)$ with $\Delta t =
\max_j \Delta t_j$. The index of the next reaction to fire is then $l =
\text{argmin}(\Delta t_j)$, from which the system state and propensities may be
updated and the global time incremented by  $\Delta t_l$. The next internal
time for reaction $l$ to fire is then computed as $\tau_l^+ = \tau_l^+ + \xi$,
where $\xi$ represents the duration between events in a Poisson 
process; the latter implies $\xi$ is exponentially distributed. Each $\tau_j$ where $j \neq l$,
corresponding to an internal time that has not reached firing, is given the
approximate update, $\tau_j = \tau_j + a_j \Delta t_l$, which is discussed in
detail in \cite[Section 4]{anderson2007modified}. An outline of the full NRM
algorithm for a general reaction network is given in Algorithm \ref{alg:NRM}.

\section{Global sensitivity analysis for stochastic models}
\label{sec:gsa}

In this section, we study convergence of sensitivity indices corresponding to
stochastic models to their deterministic counterparts.  
In Section~\ref{sec:setup}, we describe the underlying 
probabilistic setup and global sensitivity analysis via Sobol' indices.  
In
Section~\ref{sec:conv_basic}, we present a generic result regarding convergence
of the Sobol' indices of a family of random processes. Then, in
Section~\ref{sec:conv_chemkin}, we show how the generic convergence result can
be applied to stochastic chemical systems.

\subsection{The basic setup}\label{sec:setup}
Stochastic models with uncertain parameters present two sources of
uncertainties: intrinsic uncertainty due to stochasticity of the system and
uncertainty in model parameters. 

We denote the probability space carrying intrinsic stochasticity of the
system by $(\Omega, \mathcal F, \nu)$, where $\Omega$ is the sample
space equipped with a sigma-algebra $\mathcal{F}$ and a probability measure
$\nu$. 
In stochastic chemical systems, the uncertain model parameters of interest are
the reaction rates constants, $k_1, \ldots, k_M$. We model 
these as independent uniformly distributed random variables. 
Following common practice, we parameterize the uncertainty in $k_i$'s 
using a random vector $\boldsymbol\theta =
\left[\theta_1, \ldots, \theta_M\right]^\top$ whose entries are independent
$U(-1, 1)$ random variables. For example, if $k_i \sim U(a_i, b_i)$, 
then $k_i(\theta_i) = \frac12(a_i+b_i) +
\frac12(b_i-a_i)\theta_i$.

The uncertain parameter vector $\boldsymbol\theta$ takes values in $\Theta =
[-1, 1]^M$. It is convenient to work with the probability space
$(\Theta, \mathcal E,\lambda)$ for the uncertain parameters, where $\mathcal E$
is the Borel sigma-algebra on $\Theta$ and $\lambda$ is the law of
$\boldsymbol\theta$, $\lambda(d\boldsymbol\theta) = 2^{-M} d\boldsymbol\theta$.
The present setup can be easily extended to cases where $\theta_i$'s are 
independent random variables belonging to other suitably chosen distributions.
Note also that one can have additional uncertain parameters in a chemical system.

We use Sobol' indices~\cite{sobol,sobol93,saltelli2010} 
to characterize the sensitivity of a quantity of interest (QoI) to input parameter
uncertainties. For example, let $f(\boldsymbol\theta)$ be 
a scalar-valued QoI defined in terms
of the solution of the RREs corresponding to a chemical system. The first order
Sobol' indices corresponding to
$f(\boldsymbol\theta)$ are 
\begin{equation} \label{equ:det_sobol}
S_j(f) := \frac{\mathbb{V}[\mathbb{E}[f(\boldsymbol\theta) ~|~ \theta_j]]}
               {\mathbb{V}[f]}, \quad j = 1, \ldots, M.
\end{equation}
These indices quantify the proportion of the QoI variance due to the $j$th
input parameter. Here $\mathbb{E}[f(\boldsymbol\theta) ~|~
\theta_j]$ indicates conditional expectation and $\mathbb{V}[f]$ denotes the variance
of $f$. 
For further details on theory and computation methods for Sobol' indices
we refer the readers to~\cite{sobol,sobol93,saltelli2010,ralph}.

\subsection{Convergence of stochastic Sobol' indices}\label{sec:conv_basic}
We consider a family of stochastic processes 
$\{f_V(\boldsymbol\theta, \omega)\}_{V > 0}$ with 
\[
     f_V(\boldsymbol\theta, \omega) : \Theta \times \Omega \to \R,
\]
which, as discussed below,  are assumed to admit a deterministic limit as $V \to \infty$. 
The Sobol' indices corresponding to $f_V(\boldsymbol\theta, \omega)$ are
\begin{equation} \label{equ:stoch_sobol}
S_j(f_V(\cdot, \omega)) := \frac{ \mathbb{V}[\mathbb{E}[f_V(\boldsymbol\theta, \omega) ~|~ \theta_j]]  }{ \mathbb{V}[f_V(\boldsymbol\theta, \omega)]}, \quad j = 1, \ldots, M.
\end{equation}
The following result concerns the convergence of these 
indices in the limit as 
$V \to \infty$.

\begin{theorem}\label{thm:conv}
Assume
\begin{enumerate}
\item There exists
$f \in L^2(\Theta, \mathcal E,\lambda)$ such that,
for almost all $\omega \in \Omega$, 
\begin{equation}\label{equ:cond_one}
f_V(\boldsymbol\theta, \omega) \to f(\boldsymbol\theta), \quad \text{as }  V \to \infty,
\quad \text{for all } \boldsymbol\theta \in \Theta.
\end{equation}

\item For almost all $\omega \in \Omega$, $f_V(\boldsymbol\theta, \cdot)$ is 
$\mathcal{E}$-measurable and there 
exists $\varphi_{\omega}(\boldsymbol\theta) \in  L^2(\Theta, \mathcal E, \lambda)$ such
that for all $\boldsymbol\theta \in \Theta$,
\begin{equation}\label{equ:cond_two}
|f_V(\boldsymbol\theta, \omega)| \leq \varphi_{\omega}(\boldsymbol\theta), \quad \text{for all }
V > 0.
\end{equation}
\end{enumerate}
Then the stochastic Sobol' indices satisfy, 
\[
S_j(f_V(\cdot, \omega)) \to S_j(f), \quad 
\text{as } V \to \infty,\quad \nu\text{-almost surely}.
\]
\end{theorem}
\begin{proof}
By the assumptions of the theorem, there exists
a set $F \in \mathcal{F}$ with $\nu(F) = 1$ such that the
conditions~\eqref{equ:cond_one} and~\eqref{equ:cond_two} hold 
for every $\omega \in F$. By~\eqref{equ:cond_two},
we observe that $f_V(\boldsymbol\theta, \omega) \in L^2(\Theta, \mathcal{E}, \lambda)$, 
for every $\omega \in F$ and $V > 0$. Thus,
we can define the Stochastic Sobol' indices~\eqref{equ:stoch_sobol} 
for $\{f_V(\cdot, \omega)\}_{V > 0}$, for every $\omega \in F$. 

To show that
$f_V(\boldsymbol\theta, \omega) \to f(\boldsymbol\theta)$ in $L^2(\Theta, \mathcal{E}, \lambda)$, we note that for every $\omega \in F$
$|f_V(\boldsymbol\theta, \omega) - f(\boldsymbol\theta)|^2 \to 0$ pointwise in $\Theta$ and 
\[
|f_V(\boldsymbol\theta, \omega) - f(\boldsymbol\theta)|^2 \leq 
4\,\varphi_{\omega}(\boldsymbol\theta)^2 \in L^1(\Theta, \mathcal{E}, \lambda).
\]
Therefore, invoking the Lebesgue Dominated Convergence Theorem, 
we have that for all $\omega \in F$, 
$\int_{\Theta} |f_V(\boldsymbol\theta, \omega) - f(\boldsymbol\theta)|^2 \lambda(d\boldsymbol\theta) \to 0$
and thus for every $\omega \in F$
\begin{equation*}\label{equ:moments} 
\lim_{V \to \infty} \int_\Theta 
[f_V(\boldsymbol\theta, \omega)]^r \lambda(d\boldsymbol\theta) = 
\int_{\Theta} [f(\boldsymbol\theta)]^r
\lambda(d\boldsymbol\theta), \quad r = 1,2.
\end{equation*}
The convergence of the first and second moments 
of $f_V(\cdot, \omega)$ clearly implies 
\begin{equation*}\label{equ:var_conv}
\lim_{V \to \infty}
\mathbb{V}(f_V(\cdot, \omega))
= \mathbb{V}(f(\cdot)), \quad \text{for all } \omega \in F.
\end{equation*}
To finish the proof of the theorem,  
we need to show 
\begin{equation*}\label{equ:conv_cond}
   \lim_{V \to \infty} \mathbb{V}\{ \mathbb{E}(f_V(\cdot, \omega) | \theta_j)\}
   = \mathbb{V}\{ \mathbb{E}(f(\cdot) | \theta_j)\}, \quad 
   \text{for all } \omega \in F, \, j = 1, \ldots, M.
\end{equation*}
Using the reverse triangle inequality and Jensen's inequality we observe
\begin{equation*}\label{equ:inequalities}
\begin{aligned}
\left|\| \mathbb{E}(f_V(\cdot, \omega) | \theta_j) \|_{L^2(\Theta)}
- 
\| \mathbb{E}(f(\cdot) | \theta_j) \|_{L^2(\Theta)}
\right|
&\leq 
\| \mathbb{E}(f_V(\cdot, \omega) | \theta_j) - 
\mathbb{E}(f(\cdot) | \theta_j)
\|_{L^2(\Theta)}\\
&= 
\| \mathbb{E}(f_V(\cdot, \omega) - f(\cdot)| \theta_j) 
\|_{L^2(\Theta)}\\
&\leq 
\| f_V(\cdot, \omega) - f(\cdot)\|_{L^2(\Theta)},
\end{aligned}
\end{equation*} 
and thus, for all $\omega \in F$
\[
\lim_{V\to\infty} \|\mathbb{E}(f_V(\cdot, \omega) | \theta_j) \|_{L^2(\Theta)} = 
\|\mathbb{E}(f(\cdot) | \theta_j) \|_{L^2(\Theta)}.
\]
Since 
\begin{equation*}\label{equ:tower}
\begin{aligned}
\mathbb{V}\{ \mathbb{E}(f_V(\cdot, \omega) | \theta_j)\} &= 
\mathbb{E}\{ \mathbb{E}(f_V(\cdot, \omega) | \theta_j)^2 \}
- \mathbb{E}\{ \mathbb{E}(f_V(\cdot, \omega) | \theta_j) \}^2 
\\
&=\| \mathbb{E}(f_V(\cdot, \omega) | \theta_j) \|_{L^2(\Theta)}^2 
- \mathbb{E}\{ f_V(\cdot, \omega) \}^2,
\end{aligned}
\end{equation*}
we have, for all $\omega \in F$, 
\begin{equation}\label{equ:cond_exp_cov}
\lim_{V\to\infty}
\mathbb{V}\{ \mathbb{E}(f_V(\cdot, \omega) | \theta_j)\}
= \|\mathbb{E}(f(\cdot) | \theta_j) \|_{L^2(\Theta)} - \mathbb{E}\{ f(\cdot) \}^2
= 
\mathbb{V}\{ \mathbb{E}(f(\cdot) | \theta_j)\}.
\end{equation}
This, along with the convergence of the (unconditional) variance implies
\[
\lim_{V \to \infty}
S_j(f_V(\cdot, \omega)) 
= 
\lim_{V \to \infty}
\frac{ \mathbb{V}\{\mathbb{E}(f_V(\boldsymbol\theta, \omega) | \theta_j)\}  }
{ \mathbb{V}\{f_V(\boldsymbol\theta, \omega)\}}
=
\frac{ \mathbb{V}\{\mathbb{E}(f | \theta_j)\}  }
{ \mathbb{V}\{f\}}
= S_j(f), 
\] for all $\omega \in F$, $j = 1, \ldots, M$.
\end{proof}

\begin{rem}\label{rmk:SU}
A slight modification of the proof of
Theorem~\ref{thm:conv} 
leads to a more general result: namely, we can obtain 
almost sure convergence of 
the indices, 
\begin{equation}\label{equ:stoch_sobolU} 
S_U(f_V(\cdot, \omega)) := \frac{ \mathbb{V}[\mathbb{E}[f_V(\boldsymbol\theta, \omega) ~|~ 
\boldsymbol\theta_U]]  }{ \mathbb{V}[f_V(\boldsymbol\theta, \omega)]  },
\end{equation}
where 
$U = \{j_1, j_2, \ldots, j_s\} \subseteq \{1, 2, \ldots, M\}$ and $\boldsymbol\theta_U = 
\begin{bmatrix} \theta_{j_1} & \theta_{j_2} & \cdots & \theta_{j_s}\end{bmatrix}^\top$,
to $S_U(f(\cdot))$. 
\end{rem}

We recall the \emph{total} Sobol' indices~\cite{saltelli2010}, 
\begin{equation}\label{equ:stoch_soboltot}
T_j(f_V(\cdot, \omega)) := 
\sum_{U \ni j} S_U(f_V(\cdot, \omega)),
\quad
j = 1, \ldots, M.
\end{equation}
These indices quantify the relative contribution 
of $\theta_j$ by itself, and through its interactions 
with the other coordinates of of $\boldsymbol\theta$, 
to the variance of $f_V(\cdot, \omega)$. In view of 
Remark~\ref{rmk:SU}, under 
the conditions of Theorem~\ref{thm:conv} 
\[
    \lim_{V \to \infty} T_j(f_V(\cdot, \omega)) = T_j(f(\cdot)), 
    \quad \text{for almost all } \omega \in \Omega, \, j = 1, \ldots, M. 
\]

\subsection{Application to stochastic chemical kinetics}
\label{sec:conv_chemkin}

Consider the (concentration based) 
state vector $\mathbf Z^V(t, \boldsymbol\theta, \omega)$ of a stochastic 
chemical system and its deterministic counterpart $\mathbf Z(t, \boldsymbol\theta)$,
corresponding the thermodynamic limit.
Recall that $\boldsymbol\theta \in \Theta$ parameterizes the uncertainty in 
reaction rate constants.
In the present work, 
we focus on a scalar time-independent QoI
$G(\mathbf Z^V(t, \boldsymbol\theta, \omega))$ and  
its deterministic counterpart $G(\mathbf Z(t, \boldsymbol\theta))$.
Specifically, $G$ takes a vector function $\mathbf z(t)$ and returns a 
scalar QoI. 
Examples include 
\begin{subequations} \label{eq:qoi_examples}
\begin{align}
     G(\mathbf z(t)) &= z_i(t^*), \quad \text{for  fixed } t^* \in [0, T] \text{ and } i \in \{1,\ldots,N\}, \quad \text{or} 
     \\
     G(\mathbf z(t)) &= \frac1T \int_0^T z_i(t) \, dt
     \quad \text{for a fixed } i \in \{1,\ldots,N\}. 
\end{align}
\end{subequations}
In general, we assume 
$G:L^\infty([0, T]; \mathbb{R}^N) \to \R$ to be a continuous function. Note that
$L^\infty([0, T]; \mathbb{R}^N)$ is equipped with norm
$\| \cdot \|_\infty$ given by $\| \mathbf{z} \|_\infty = \sup_{t \in [0, T]} \|z_i(t)\|_2$. 

To put things in the
notation of the previous subsection, we consider
\[
   f_V(\boldsymbol\theta, \omega) = G(\mathbf Z^V(t, \boldsymbol\theta, \omega)),
   \quad \boldsymbol\theta \in \Theta, \omega \in \Omega,
\]
and the corresponding limiting (deterministic) quantity, $f(\boldsymbol\theta)
=  G(\mathbf Z(t, \boldsymbol\theta))$.
Note that by~\eqref{kurtz}, for fixed $\boldsymbol \theta \in \Theta$, as $V \to \infty$
\[
\|\mathbf Z^V(\cdot, \boldsymbol\theta, \omega) - \mathbf Z(\cdot, \boldsymbol\theta)\|_\infty
\to 0, \quad \text{for almost all } \omega \in \Omega.
\]
Therefore, by the Continuous Mapping Theorem, see e.g., \cite{durrett}, 
for each $\boldsymbol \theta \in \Theta$, 
\begin{equation}\label{equ:conv_kurtz}
f_V(\boldsymbol\theta, \omega) \to f(\boldsymbol\theta), \quad \text{almost surely}, 
\end{equation}
as $V \to \infty$.
We consider the convergence of the stochastic Sobol' indices 
$S_j(f_V(\cdot, \omega))$ to their deterministic counterparts 
$S_j(f(\cdot))$, $j = 1, \ldots, M$, as $V \to \infty$, i.e., in the
thermodynamic limit. 
Here we discuss how things can be put in the framework of Theorem~\ref{thm:conv}, 
which would then imply almost sure convergence of the stochastic Sobol' indices
to their limiting deterministic counterparts. 

Theorem~\ref{thm:conv} requires existence of a set of full measure in $\Omega$ such that
the convergence in~\eqref{equ:conv_kurtz} holds. To ensure this, we consider a modification 
of $f_V(\boldsymbol\theta, \omega)$ as follows. We know that for each $\boldsymbol\theta \in 
\Theta$, there exists a set of full measure $F_\theta \subseteq \Omega$ for which
the convergence~\eqref{equ:conv_kurtz} holds. Define 
\[
     \tilde{f}_V(\boldsymbol\theta, \omega) = \begin{cases}
     f_V(\boldsymbol\theta, \omega) \quad \text{ if } \omega \in F_\theta,\\
     f(\boldsymbol\theta) \quad \text{otherwise}.
     \end{cases}
\]
Note that, we have $\nu\left( \{ \omega \in \Omega : \tilde{f}_V(\boldsymbol\theta, \cdot) =  
f_V(\boldsymbol\theta, \omega)
\}\right) = 1$, 
for every $\boldsymbol\theta \in \Theta$. That is $\tilde{f}_V(\boldsymbol\theta, \cdot)$
is a modification of $f_V(\boldsymbol\theta, \cdot)$. 
Note that this modification satisfies the following: for every $\omega \in \Omega$,
$\tilde{f}_V(\boldsymbol\theta, \omega) \to f(\boldsymbol\theta)$ for all $\boldsymbol 
\theta \in \Theta$.
With a slight abuse of notation,
we will denote this modification by $f_V(\boldsymbol\theta, \omega)$ from this point on.
To ensure that Theorem~\ref{thm:conv} applies, we need also the boundedness 
assumption~\eqref{equ:cond_two}. 

To discuss the boundedness assumption \eqref{equ:cond_two}, we take 
a step back and first discuss conditions ensuring boundedness of 
the stochastic system trajectory $\{\mathbf Z^V(t,
\boldsymbol\theta, \omega)\}_{V > 0}$. 
Consider the state vector $\mathbf{X}^V(t)$.
Non-negativity of
this state vector requires the propensity functions
to be proper~\cite{rathinam2015moment}: for $j = 1, \ldots, M$,
we assume  
for all $\mathbf{x} \in \mathbb{Z}_+^N$, if $\mathbf{x} +
\boldsymbol\nu_j \notin \mathbb{Z}_+^N$, then $a_j^V(\mathbf{x}) = 0$. 
Boundedness of components of $\mathbf X^V(t)$ requires further (mild)
assumptions, as formalized in~\cite[Theorem 2.8 and 2.11]{rathinam2015moment}.
Interestingly, the only requirements concern the stoichiometric matrix
$\boldsymbol \nu$. Namely, assuming the existence of 
a vector $\boldsymbol \alpha \in \mathbb{Z}^N_{\geq 0}$ such that 
$\boldsymbol \alpha^\top \boldsymbol \nu \leq 0$ and 
$\alpha_i > 0$ is necessary and sufficient for boundedness of 
$X_i^V(t)$. Specifically, if such $\boldsymbol\alpha$ exists, 
$\boldsymbol \alpha^\top \mathbf{X}^V(t) = \boldsymbol\alpha^\top
(\mathbf{X}^V(0) + \boldsymbol\nu \mathbf{R}(t)) \leq 
\boldsymbol\alpha^\top \mathbf{X}^V(0)$. Therefore, 
\[
X_i^V(t) \leq (1/\alpha_i) \boldsymbol\alpha^\top \mathbf{X}^V(0)
=(V/\alpha_i) \boldsymbol\alpha^\top\mathbf{x}_0.
\]
Thus, in terms of concentrations 
\[
    Z^V_i(t) = X_i^V / V \leq (1/\alpha_i) \boldsymbol\alpha^\top\mathbf{x}_0. 
\]

Therefore, 
we have that the $i$th component of $\mathbf{Z}^V$ remains uniformly bounded by 
$(1/\alpha_j) \boldsymbol\alpha^\top \mathbf{x}_0$. Moreover, this bound 
is independent of the reaction rate constants, i.e., independent of $\boldsymbol\theta$.
Thus, if a vector $\boldsymbol\alpha$ satisfying the aforementioned properties
exists for all the components of the state vector, then the concentration based
state vector $\mathbf Z^V$ remains uniformly bounded by a constant.   
In fact, we need to only ensure boundedness of the components of
$\mathbf Z^V$ that appear in definition of $G$. 
Given the function $G$,
which defines the QoI, is sufficiently well-behaved, one may argue that $f_V$
inherits the boundedness necessary to satisfy \eqref{equ:cond_two}. For
example, if $G$ is defined as in \eqref{eq:qoi_examples}, then establishing
boundedness of $\{Z_i^V(t, \boldsymbol\theta, \omega)\}_{V > 0}$ is
sufficient to satisfy \eqref{equ:cond_two} for the QoI, $f_V$.

\section{Numerical results}
\label{sec:numerical_results}
In light of the convergence properties exhibited by stochastic chemical
reaction systems, we aim to demonstrate numerically the results of
Theorem~\ref{thm:conv}. Convergence results will be presented first for the
Michaelis--Menten reaction system and then for a higher-dimensional example
arising from the study of genetic networks. Attention will also be devoted to the
computation of Sobol' indices and the random sampling necessary to compute the
stochastic Sobol' indices introduced in Section~\ref{sec:gsa}.

\subsection{The  Michaelis--Menten system}
The Michaelis--Menten reaction is the most well-known example of enzymatic catalysis in the chemical kinetics literature\cite{higham,anderson,le2015variance}: 
\begin{align} \label{mm_react}
  S + E &\xrightarrow{k_1} C \nonumber \\ 
    C &\xrightarrow{k_2} S + E \\
    C &\xrightarrow{k_3} P + E \nonumber 
\end{align}
In \eqref{mm_react}, the substrate $S$ binds to the enzyme $E$ to form the
complex $C$. The complex may either dissociate back into the substrate and
enzyme or dissociate into the enzyme and a product $P$. Figure
\ref{fig:mm_traject} depicts 25 realizations of the reaction dynamics using the
NRM algorithm with a final time of $T=50$. The parameters, corresponding to the
rate constants in the propensity functions, are fixed to the nominal values
$\bar{k}_1 = 10^6$, $\bar{k}_2 = 10^{-4}$, and $\bar{k}_3 = 0.1$ provided in~\cite{wilkinson}. Figure
\ref{fig:mm_traject} depicts concentrations of each species for a system size of
$V_\text{nom} =  n_A \mathcal{V}_\text{nom}$, where the nominal volume of the reaction
system is $\mathcal{V}_\text{nom} = 10^{-15}\,\si{m^3}$.
\begin{figure}[!htb]
\centering
\includegraphics[width=.8\textwidth]{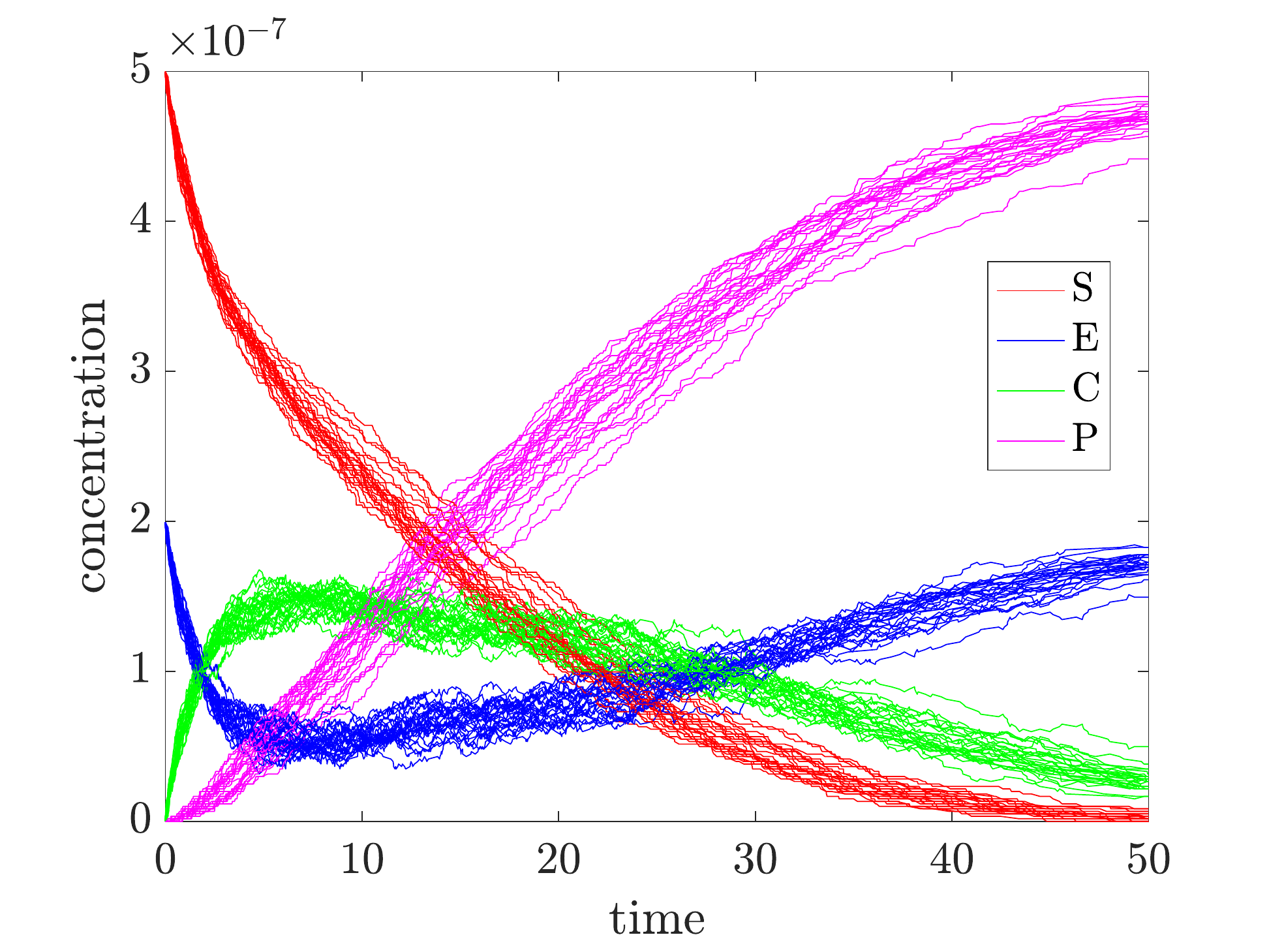}
\caption{25 realizations of  Michaelis-Menten trajectories computed via NRM with nominal parameters, varying $\omega$.} \label{fig:mm_traject}
\end{figure}

In Figure~\ref{fig:trajectconv_omega}~we illustrate convergence of the RTC
trajectories to the RRE trajectories as the system size increases. We hold the
parameters fixed to their nominal values and plot 25 realizations of the
product $P^V(t, \omega) = Z^V_4(t, \omega)$ along with the corresponding RRE 
trajectory. As the system
size increases, the ensemble of RTC trajectories converge to the RRE trajectory. In
Figure~\ref{fig:trajectconv_omega}, the quantity $m$ denotes the multiplicative factor
by which the system size is varied. For the purpose of the simulation, $m$ is
related to the system size by the relation $V = m \cdotp V_\text{nom}$.
\begin{figure}[!htb]
    \centering
    \includegraphics[trim= 30 20 20 23,width=.75\textwidth]{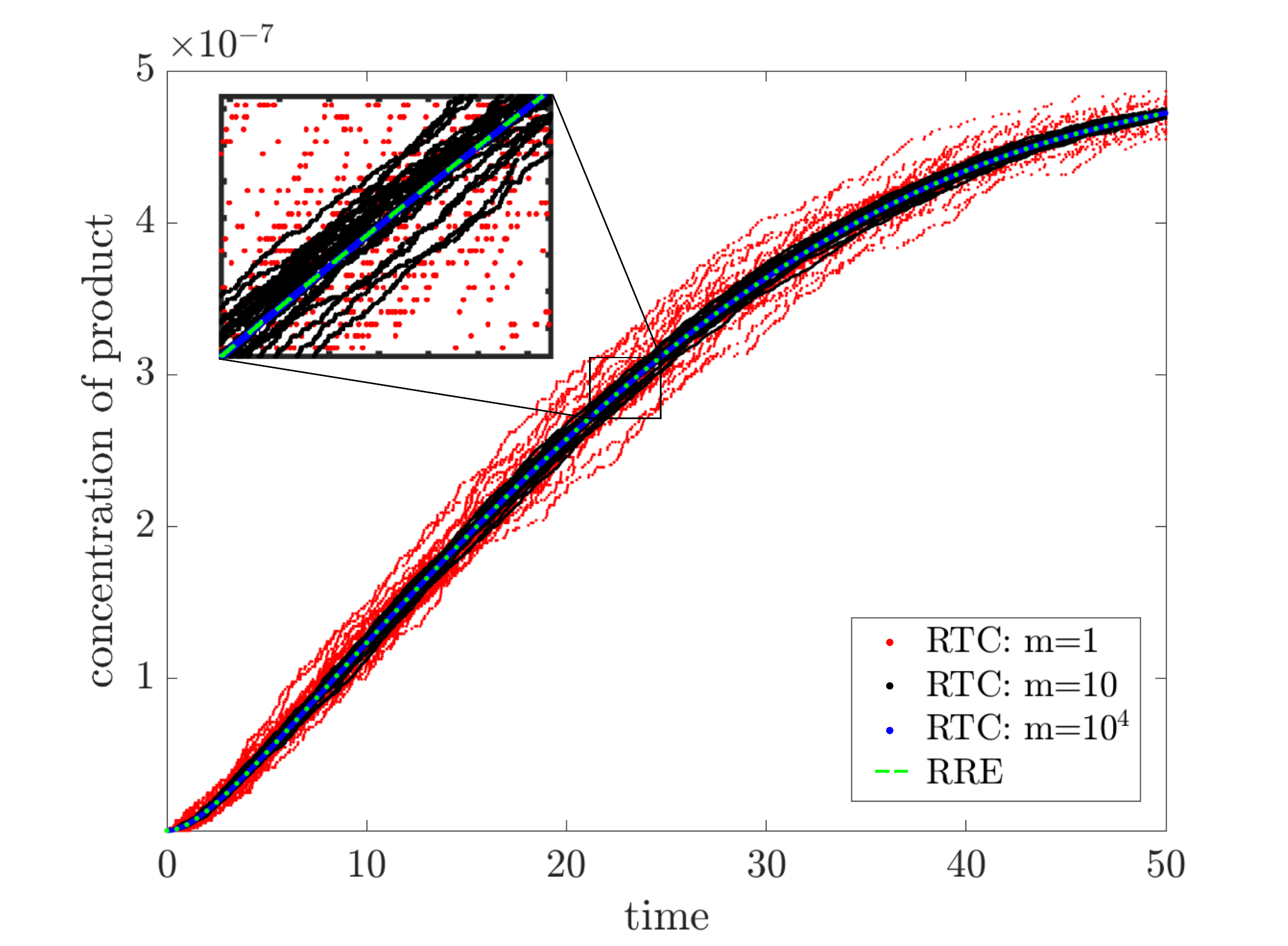}%
    \caption{Convergence of the product $P^V(t, \omega)$ the 
corresponding RRE solution at the nominal 
parameter values plotted as system size grows. 
}%
\label{fig:trajectconv_omega}%
\end{figure}

\subsubsection{The QoI}
In the present study we focus on the stochastic QoI 
\[
 f_V(\boldsymbol\theta, \omega) = \frac{1}{T}\int_0^T  Z^V_4(t;  \boldsymbol\theta, \omega ) \,dt, 
\]
where $\mathbf{Z}^V$ is
the solution of the RTC.  The corresponding deterministic QoI is 
\[
 f(\boldsymbol\theta) = \frac{1}{T}\int_0^T Z_4(t; \boldsymbol\theta)\,dt,
\]
where $\mathbf{Z}$ is computed by solving the accompanying
RRE. To get a sense of the statistical properties of the QoI, we sample $f_V$
and $f$ over the uncertain parameter domain given by $\Theta = [-1, 1]^3$, 
and with the uncertain rate constants defined as 
\[
k_i(\theta_i) = \bar{k}_i + (0.1 \bar{k}_i) \theta_i, \quad i = 1, 2, 3,
\]
where $\bar{k}_i$'s are the nominal reaction rate constants as defined above.

Figure \ref{fig:sample_qoi_pdf} shows PDFs of $f$ sampled in $\Theta$,
$f_V$ sampled in $\Theta \times \Omega$, and $f_V$ sampled in $\Omega$ while
using nominal parameters. 

All samples of $f_V$ used in Figure \ref{fig:sample_qoi_pdf} use the $V =
V_\text{nom}$.

\begin{figure} [!htb]
    \centering
    \subfloat{{\includegraphics[trim= 20 5 20 7,clip,width=.7\textwidth]{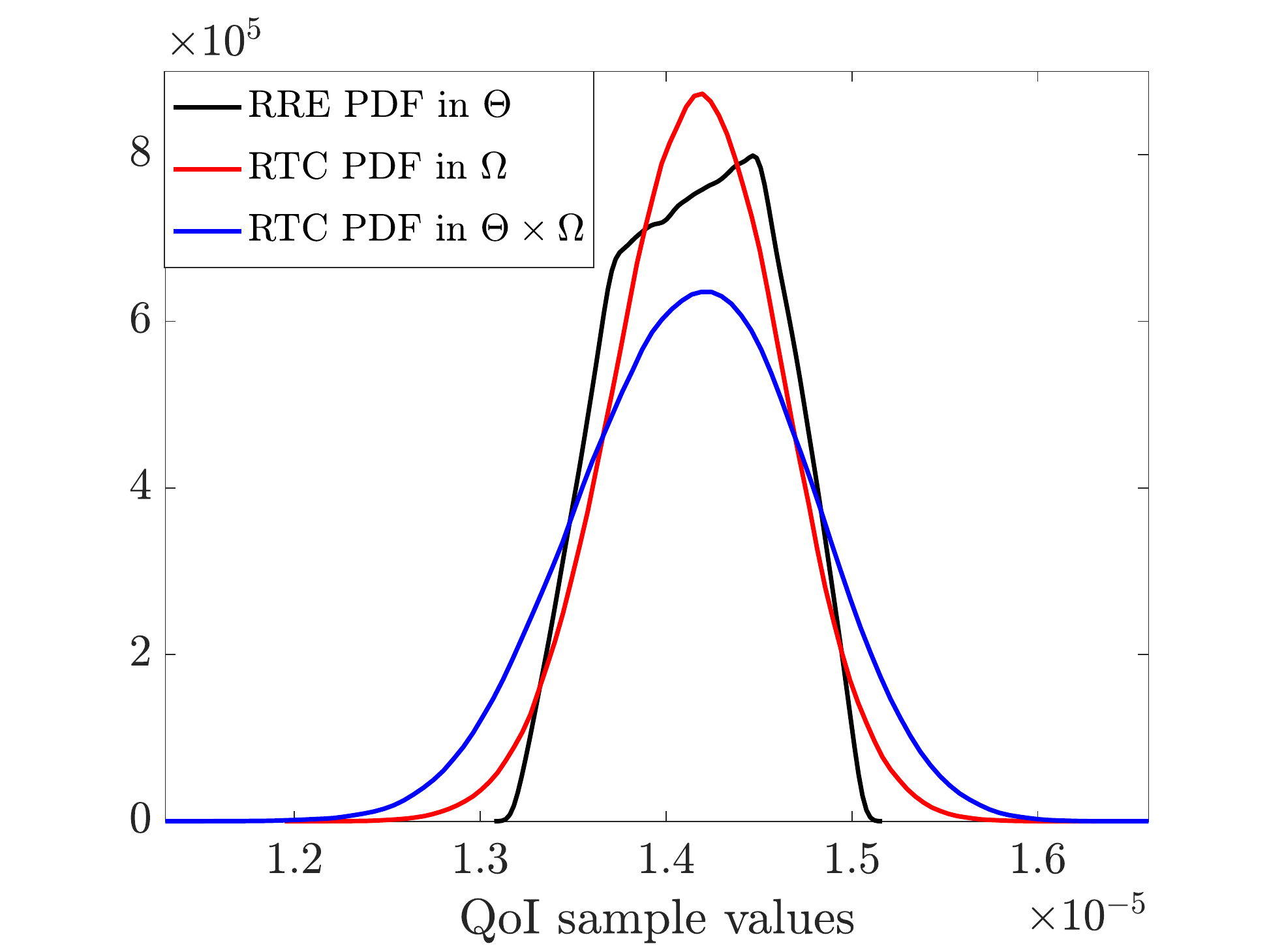} }}
    \caption{Estimated PDFs of $f_V$ sampled over $\Omega$ and $\Theta \times \Omega$ and $f$ sampled over $\Theta$, respectively.}%
    \label{fig:sample_qoi_pdf}%
\end{figure}

\subsubsection{Global sensitivity analysis}

In this section, we turn to estimating Sobol' indices in both the stochastic
and deterministic setting.  For the purpose of this study, we focus on the
computation and convergence of the total Sobol' indices. The method detailed
below can be applied to Sobol' indices of any order. 

Sobol' indices measure the relative contribution of a subset of uncertain
parameters to the variance of some QoI. Consequently, it is natural to consider
QoIs which are deterministic functions of these uncertain parameters, without
any additional variance contributed by a secondary source. When modeling
chemical systems using stochastic processes, such as the RTC, the model
parameters and internal stochasticity both provide sources of uncertainty,
which must be accounted for separately. We 
summarize the process of estimating Sobol' indices in the 
deterministic and stochastic cases in the Algorithm~\ref{alg:Sobol}, where the number of uncertain parameters is denoted $p$. Note, it is not always the case that $p
= M$, the number of reactions. 

\begin{algorithm}[h]
\caption{Sobol' indices for a chemical system with fixed system size.} 
\label{alg:Sobol}
\begin{algorithmic}[1]
\REQUIRE Method of evaluating $f_V(\boldsymbol\theta, \omega)$ and $f(\boldsymbol\theta)$, $N_s$: number of parameter samples, set of $M_s$ random seeds $\{ \xi_i \}_{i=1}^{M_s}$, system size $V$.
\ENSURE Total Sobol' indices: $\{ T_1^V(\omega_i), \dots, T_p^V(\omega_i)\}_{i=1}^{M_s}$ and $\{ T_1, \dots, T_p \}$.

\STATE Draw $N_s(p+2)$ samples uniformly in $\Theta$ \hfill
\COMMENT{see~\cite{ ralph} for details}
\STATE \verb+% stochastic indices %+ 
\FOR{$i = 1, \dots, M_s$}
\STATE Seed random number generator with $\xi_i$, corresponding to realization $\omega_i$
\FOR{$j = 1, \dots, N_s(p+2)$} 
\STATE Evaluate and store $f_V(\boldsymbol\theta_j, \omega_i)$ samples
\ENDFOR
\STATE Using $f_V$ samples, estimate Sobol' indices: $\{T_1^V(\omega_i), \dots, T_p^V(\omega_i) \}$
\ENDFOR

\STATE \verb+% deterministic indices %+
\FOR{$j = 1, \dots, N_s(p+2)$}
\STATE Evaluate and store $f(\boldsymbol\theta_j)$ samples
\ENDFOR
\STATE Using $f$ samples, estimate Sobol' indices: $\{T_1, \dots, T_p \}$
\end{algorithmic}
\end{algorithm}

In the stochastic setting, fixing a particular $\omega_i$ turns $f_V$ into a
deterministic function of the uncertain parameters. From that point, the
process of estimating Sobol' indices is identical to the deterministic case. We
estimate Sobol' indices using Monte Carlo integration, see~\cite{saltelli2010, ralph}
or~\cite[Section 4.5]{saltellibook} for details. In Algorithm~\ref{alg:Sobol}, the cost of estimating first order and total indices for each fixed $\omega_i$ is $N_s(p+2)$ evaluations of the QoI, where $N_s$ is user-defined.

The realizations of the stochastic indices
correspond to $\omega_i \in \Omega$, $i = 1, \ldots, M_s$, prescribed by the choice of
random seed. We also note that the stochastic indices are functions of the
given system size, while the deterministic indices do not depend on $V$ and
should not be recomputed each time $V$ is changed. For a fixed $V$, we may
compare the distribution of each $T_i^V$ with the deterministic value of $T_i$.

Returning to the Michaelis--Menten example, in Figure~\ref{fig:totalsob_pdfs}~we plot the PDFs of the stochastic total indices corresponding to the default $V$, where $m=1$. 
 \begin{figure} [!htb]
    \centering
    \includegraphics[trim= 35 20 35 20, clip,width=0.33\textwidth]{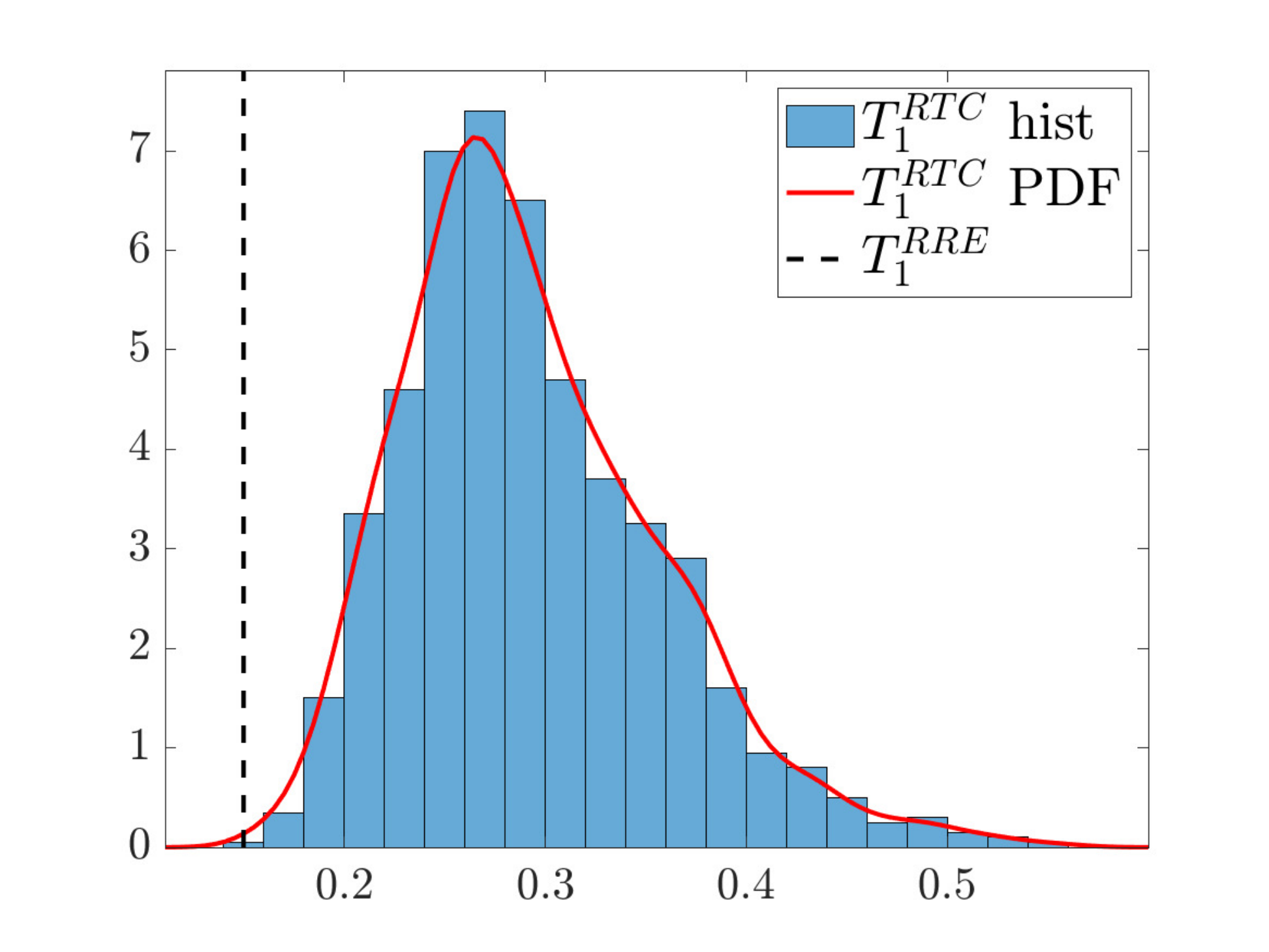}%
    \includegraphics[trim= 35 20 35 20, clip,width=0.33\textwidth]{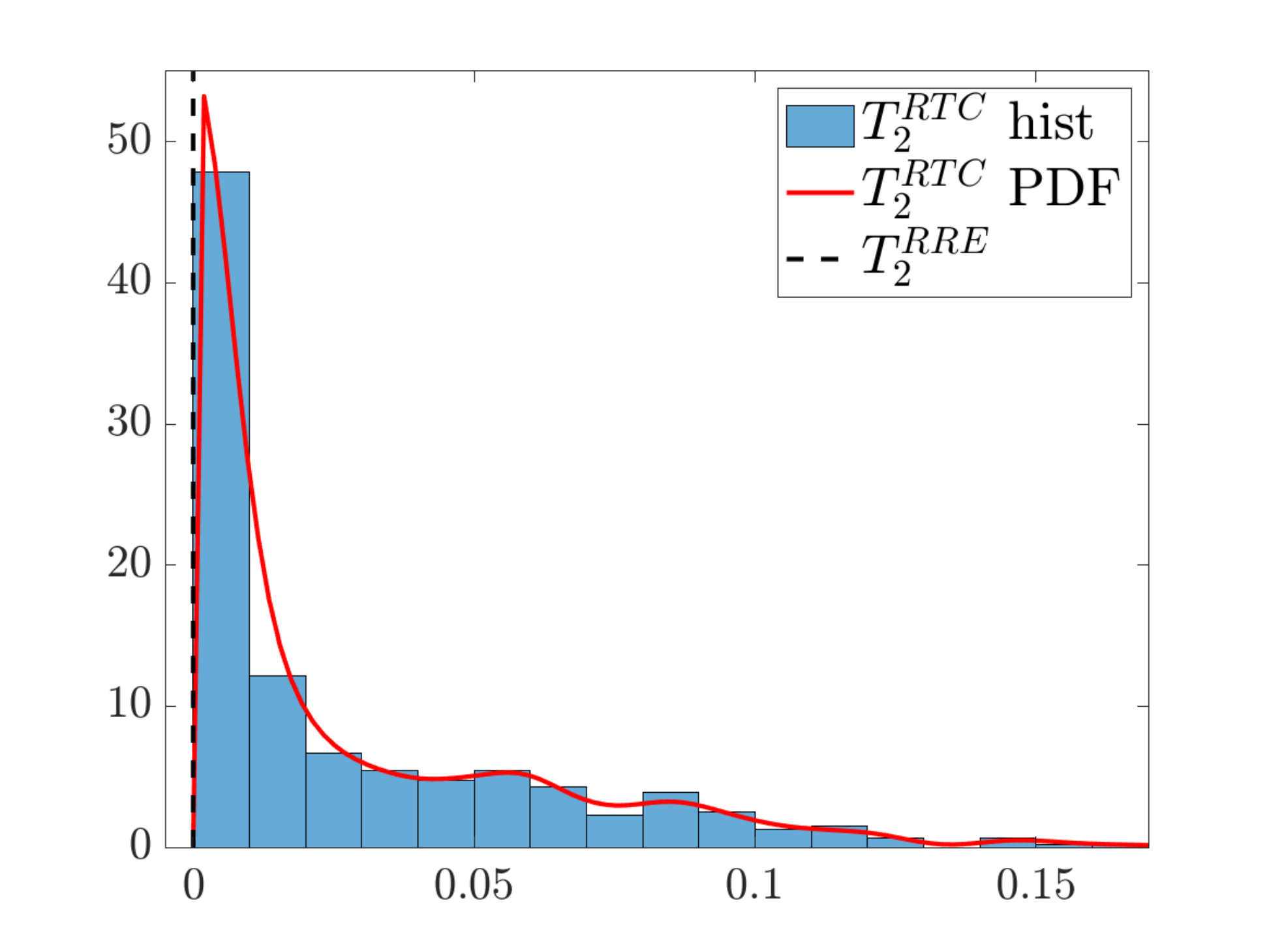} %
    \includegraphics[trim= 35 20 35 20, clip,width=0.33\textwidth]{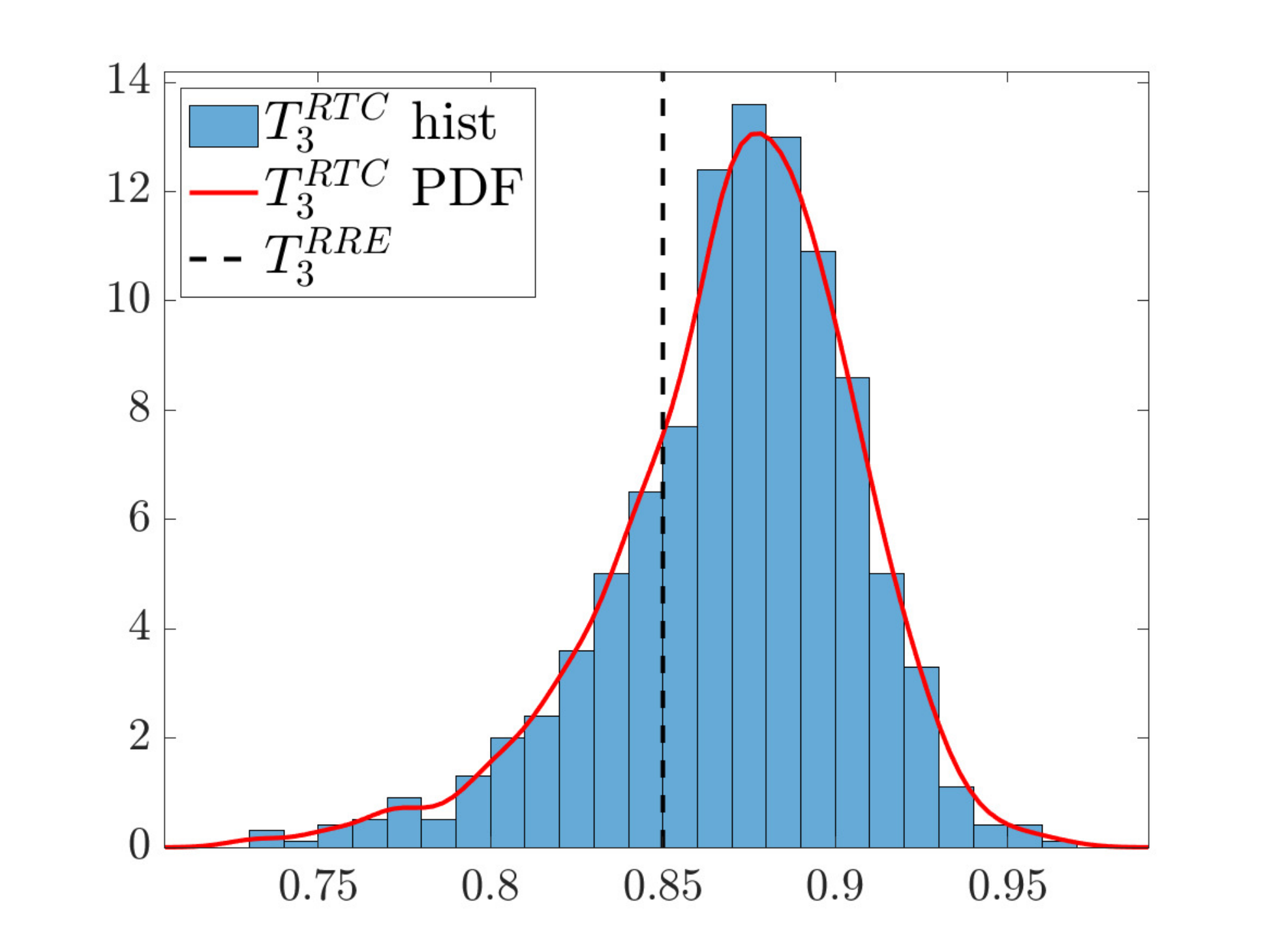}%
    \caption{Histogram and PDF estimates for the total Sobol' indices for $k_1, k_2$, and $k_3$, respectively. Black dashed lines indicate the deterministic value of the RRE total indices.}%
    \label{fig:totalsob_pdfs}%
\end{figure}
The deterministic indices, estimated with $N_s = 10^7$ samples, are $T_1
\approx 1.5 \times 10^{-1}, T_2 \approx 1.2 \times 10^{-7}$, and $T_3 \approx
8.5 \times 10^{-1}$, indicating that the third reaction, where the complex
dissociates into the enzyme and the product, is the most important and the
second reaction, where complex dissociates into the enzyme and substrate, is
the least important, contributing almost no variance. 

\subsubsection{Convergence of Sobol' indices} 
One may verify that the conditions on the QoI necessary for \ref{thm:conv} to hold are satisfied in the present case. Thus we demonstrate numerically the convergence of the stochastic Sobol' indices to the stated deterministic values. After we have computed multiple realizations of the stochastic indices at increasing, discrete values of $V$, we examine the evolution of their distribution as $V$ increases. 

 \begin{figure} [!htb]
    \centering
    \includegraphics[trim= 5 5 35 20, clip,width=0.34\textwidth]{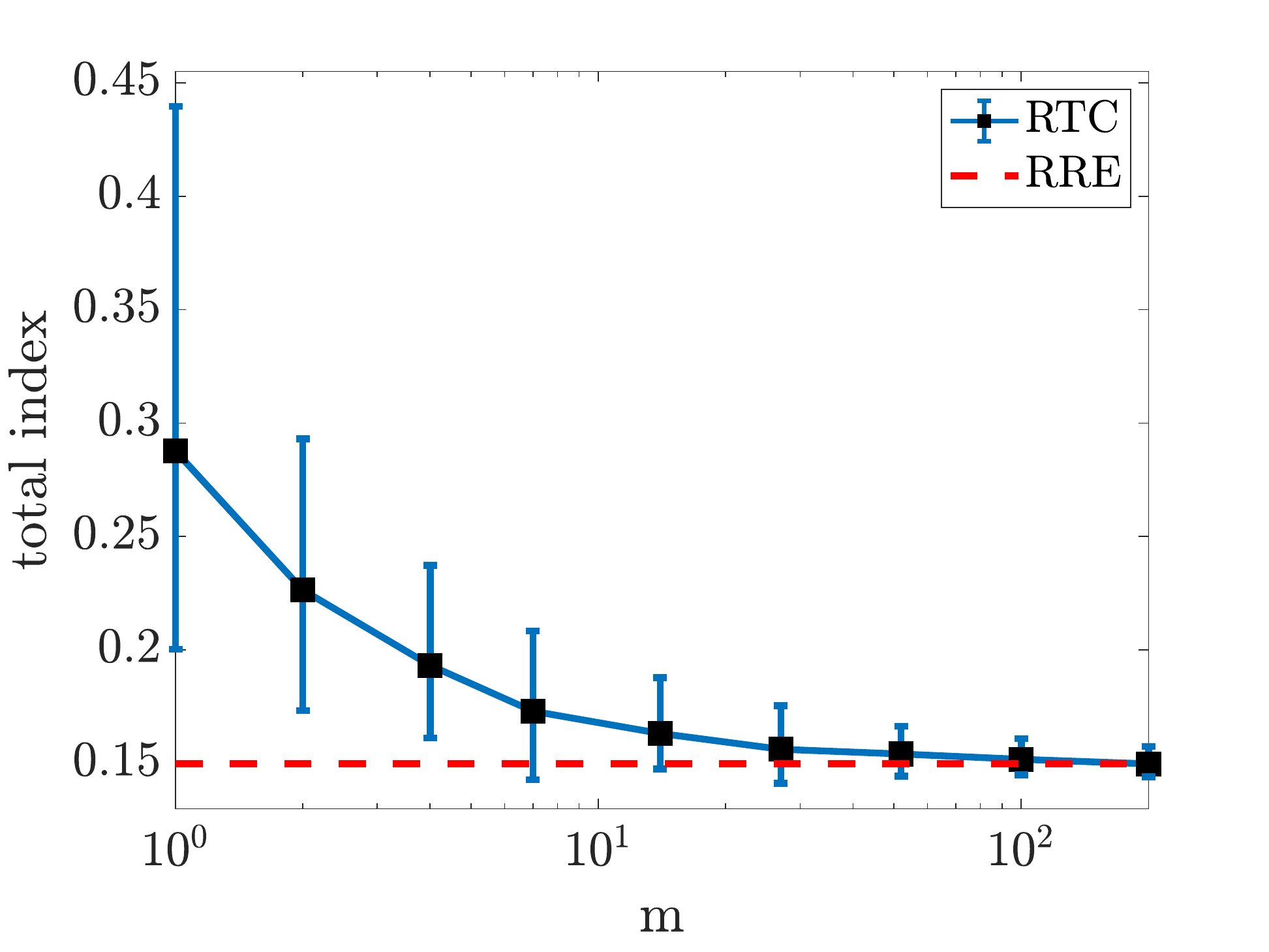}%
    \includegraphics[trim= 25 5 35 20, clip,width=0.325\textwidth]{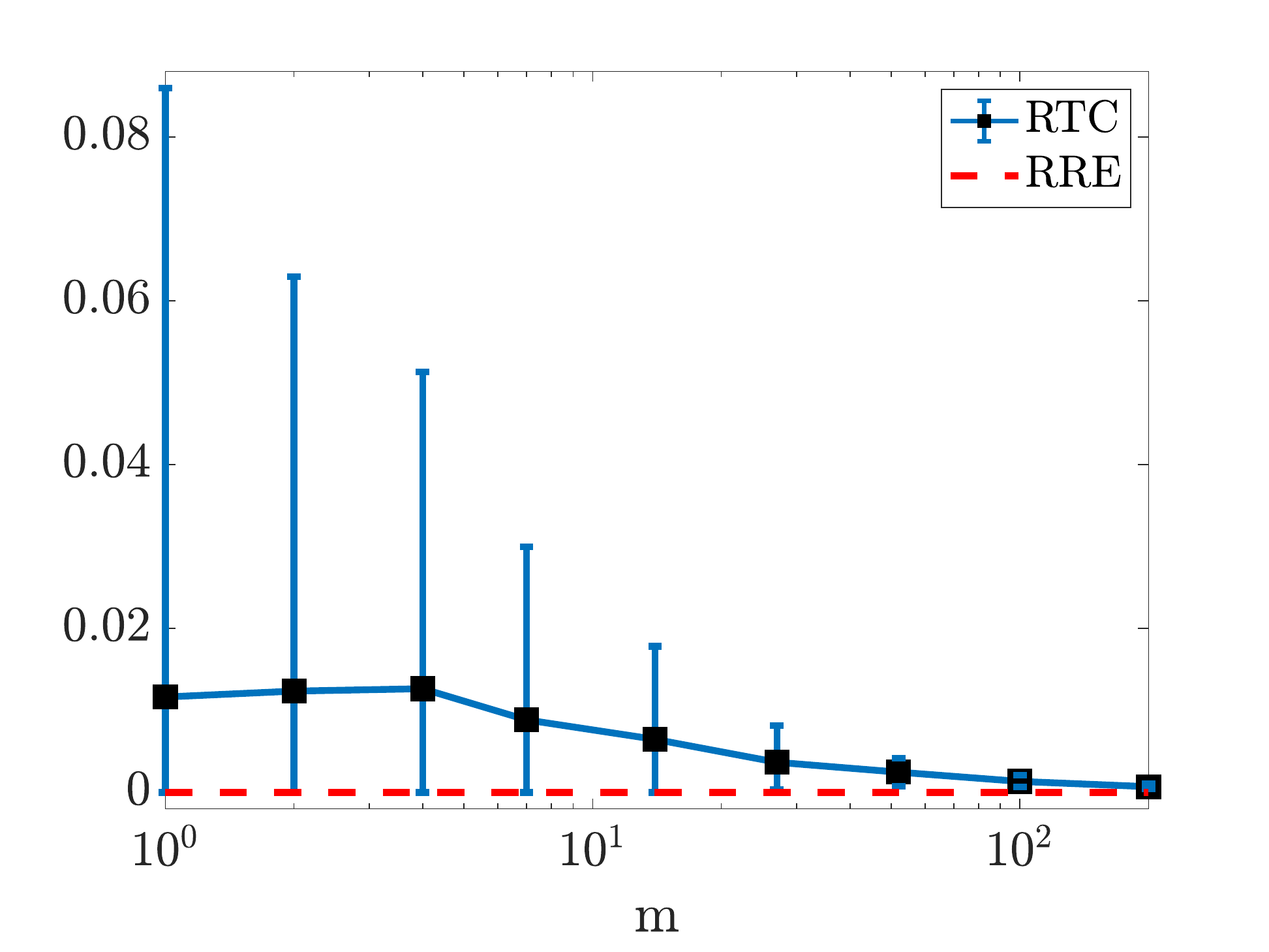} %
    \includegraphics[trim= 25 5 35 20, clip,width=0.325\textwidth]{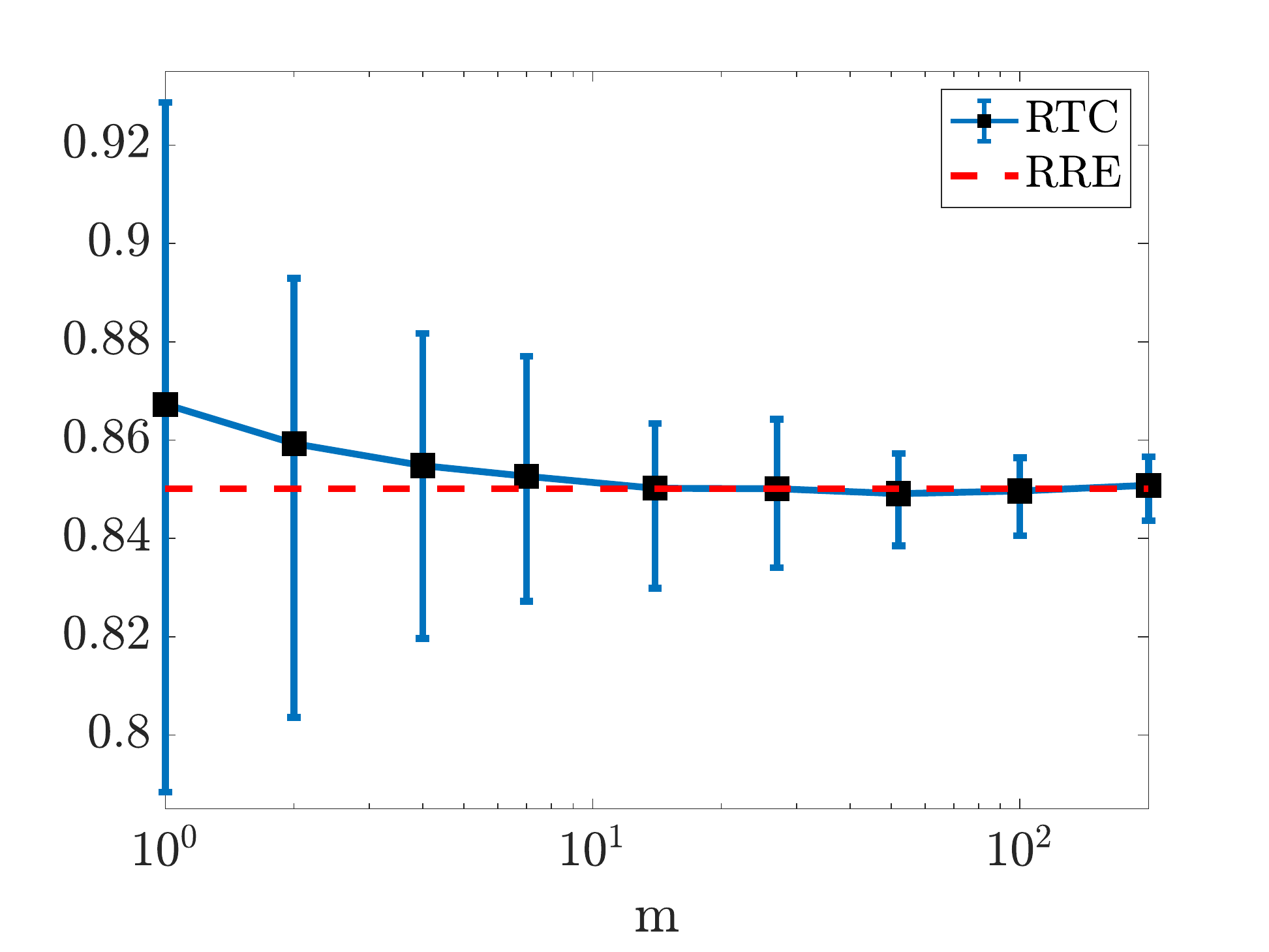}%
    \caption{Convergence of the mean total Sobol' index as a function of $V$ for parameters $k_1, k_2$, and $k_3$, respectively. Note the vertical axes of each figure are not over the same range. The lower and upper bounds of the error
bars indicate the $5$th and $95$th percentiles, respectively.}
    \label{fig:errorbar_conv}%
\end{figure}

Figure \ref{fig:errorbar_conv} demonstrates the convergence of
$\mathbb{E}[T_i^{V_m}(\omega)]$ for $i = 1, 2, 3$, 
for increasing values of system size $V_m = m V_\text{nom}$, 
$m = 1, \ldots, 200$.
The
error bars represent the $5$th and $95$th percentiles of the distribution of
stochastic indices at a particular system size. Figure \ref{fig:errorbar_conv}
suggests the convergence of the PDF for each $T_i^V(\omega)$ to a Dirac
distribution centered at the deterministic value of the Sobol' index
corresponding to the RRE. This sort of convergence may also be demonstrated for
lower order Sobol' indices, as addressed in Remark \ref{rmk:SU}.

Figure \ref{fig:histconv} gives a three-dimensional view of the convergence in Figure \ref{fig:errorbar_conv}. We plot a series of normalized histograms at specific values of $m$, converging to Dirac distributions centered at the RRE total indices. These histograms, even for two orders of magnitude difference in $V$, show a clear trend towards the limiting values given by the RRE. 

\begin{figure} [!htb]
    \centering
    \includegraphics[trim= 35 20 25 30, clip,width=0.49\textwidth]{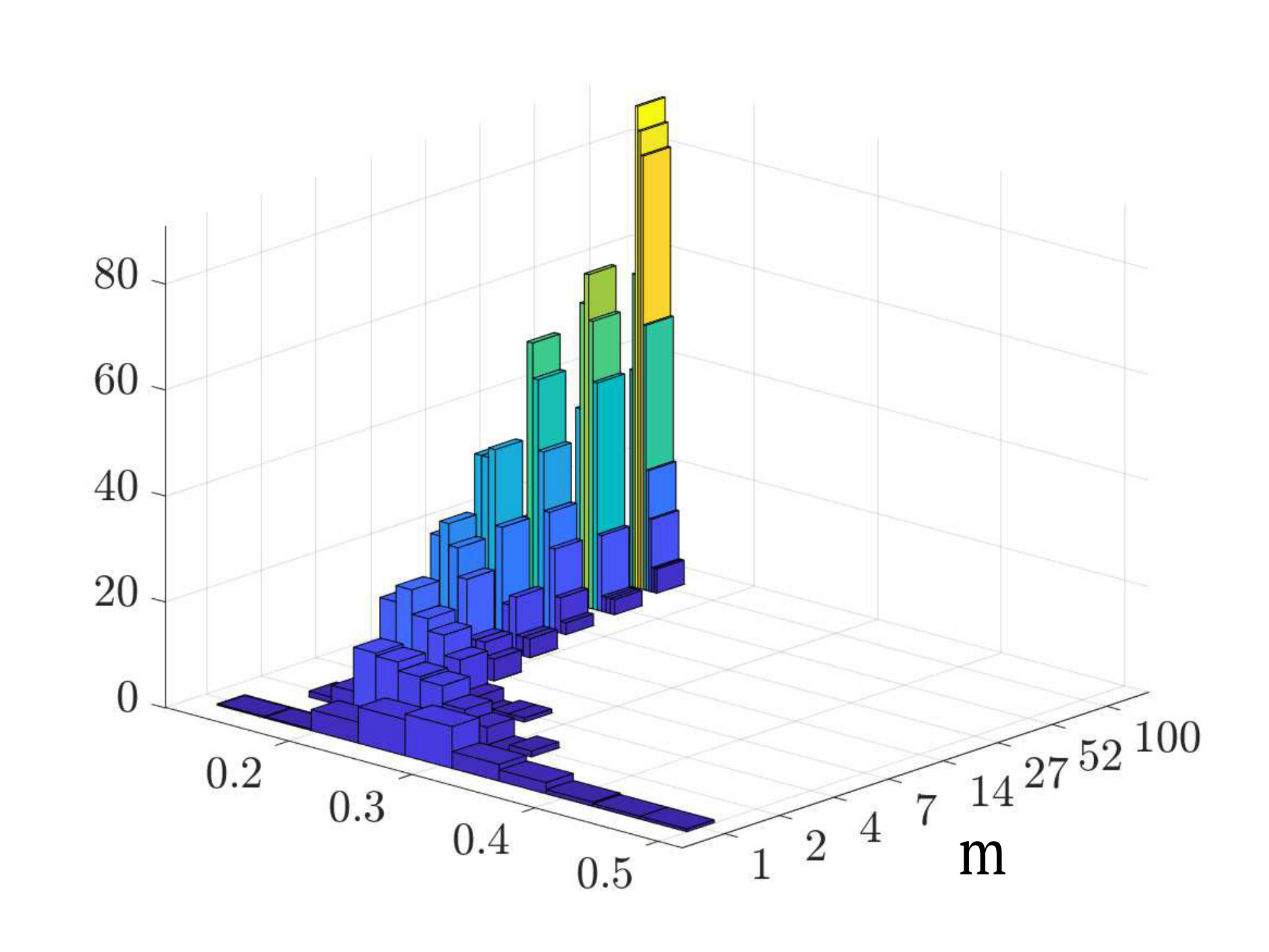} 
    \includegraphics[trim= 20 20 20 30, clip,width=0.49\textwidth]{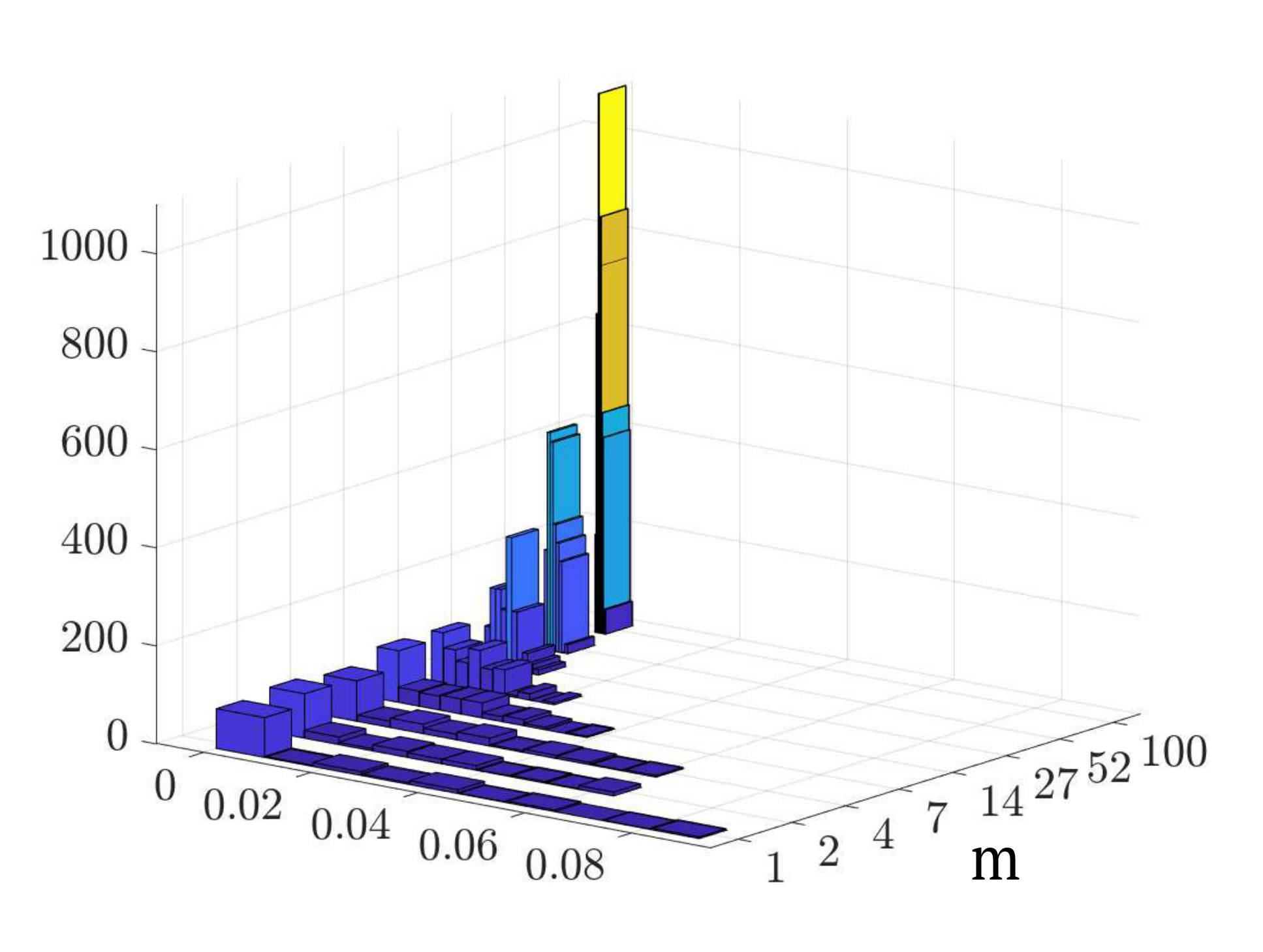} 
    \includegraphics[trim= 35 20 20 30, clip,width=0.49\textwidth]{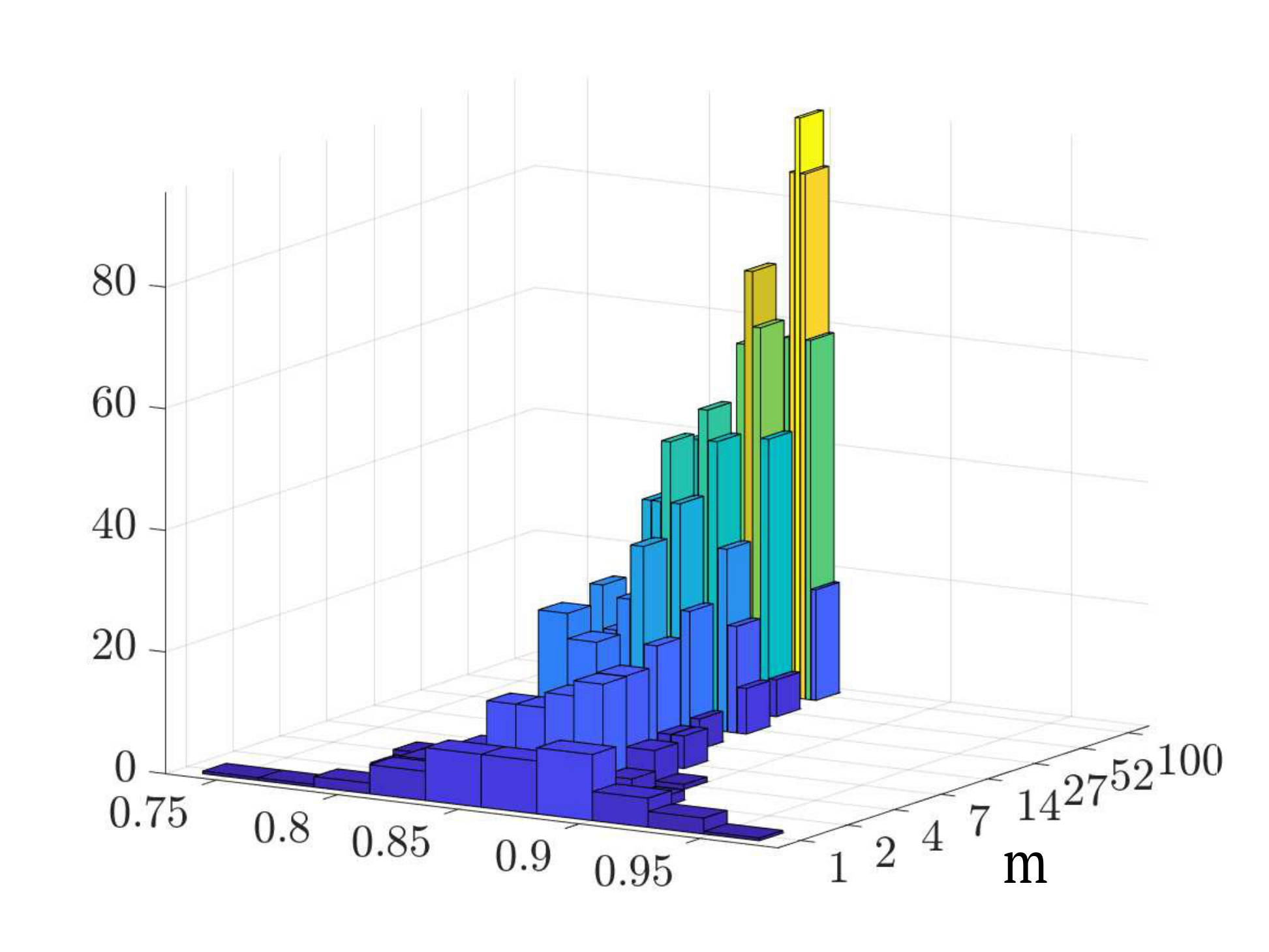}%
    \caption{Histograms at discrete $V$ values of the total Sobol' indices for $k_1, k_2$, and $k_3$, respectively. Note again that the distributions are not over the same range of values. }
    \label{fig:histconv}%
\end{figure}

Figures \ref{fig:errorbar_conv} and \ref{fig:histconv} can perhaps most
naturally be understood as illustrating the convergence in distribution of the RTC Sobol'
indices, an implication of the pointwise convergence of the PDF. In this case, $T_i^V(\omega)$ is the random
variable that converges in distribution for each $i = 1, 2, 3$ as $V$
approaches infinity. 

\subsection{The genetic oscillator system}
We next consider the genetic oscillator system presented in~\cite{vilar2002}, which models the evolution of activator and repressor proteins that govern the circadian
clocks of a wide variety of organisms. The system consists of nine species,
including genes, mRNAs, and the two proteins. We have $M = 16$ 
reactions and sixteen uncertain parameters. Following the form of the chemical system presented in~\cite{sheppard2012pathwise}, we provide the reaction diagrams, propensity functions, and nominal parameter values in Table~\ref{tab:genetic}.

\begin{table}
\label{tab:genetic}
\begin{minipage}{0.5\textwidth}
 \begin{tabular}[!htb]{||c c||}
 \hline
Reaction & Propensity Function \\ [0.5ex] 
 \hline
$P_a \to P_a + mRNA_a$ & $\alpha_AP_a$  \\ 
 $P_{a-}A \to P_{a-}A + mRNA_a$ & $\alpha_a\alpha_A P_{a-}A$  \\
 $P_r \to P_rmRNA_r$ & $\alpha_R P_r$  \\
 $P_{r-}A \to P_{r-}A + mRNA_r$  & $\alpha_r\alpha_R P_{r-}A$  \\
 $mRNA_a \to mRNA_a + A$  & $\beta_A mRNA_a$  \\
 $mRNA_r \to mRNA_r + R$  & $\beta_R mRNA_r$  \\
 $A + R \to C$  & $\gamma_C AR$  \\
 $P_a + A \to P_{a-}A$  & $\gamma_A P_a A$  \\
 $P_{a-}A \to P_a + A$  & $\theta_A P_{a-}A$  \\
 $P_r + A \to P_{r-}A$  & $\gamma_R P_r A$  \\
 $P_{r-}A \to P_r + A$  & $\theta_R P_{r-}A$  \\
 $A \to \emptyset$  & $\delta_A A$  \\
 $R \to \emptyset$  & $\delta_R R$  \\
 $mRNA_a \to \emptyset$  & $\delta_{MA} mRNA_a$  \\
 $mRNA_r \to \emptyset$  & $\delta_{MR} mRNA_r$  \\
 $C \to R$  & $\delta_A' C$  \\ [1ex] 
 \hline
\end{tabular} 
\end{minipage}
\hspace{20mm}
\begin{minipage}{0.5\textwidth}
\begin{tabular}{|| c c ||}
\hline
Parameter & Value \\
\hline
$\alpha_A$ & $50.0$ \\
$\alpha_R$ & $0.01$ \\
$\beta_A$ & $50.0$ \\
$\beta_R$ & $5.0$ \\
$\gamma_C$ & $20.0$ \\
$\gamma_A$ & $1.0$ \\
$\theta_A$ & $50.0$ \\
$\gamma_R$ & $1.0$ \\
$\theta_R$ & $1.0$ \\
$\delta_A$ & $1.0$ \\
$\delta_R$ & $0.2$ \\
$\delta_{MA}$ & $10.0$ \\
$\delta_{MR}$ & $0.5$ \\
$\delta_A'$ & $1.0$ \\
$\alpha_a$ & $10.0$ \\
$\alpha_r$ & $5000$ \\
\hline
\end{tabular}
\end{minipage}
\caption{Genetic oscillator reactions, propensity functions, 
and nominal parameter values, see~\cite{sheppard2012pathwise}.}
\end{table}

As with the Michaelis--Menten system, the RTC models the evolution of the
stochastic system and the RRE models the deterministic system, with the two
models linked by thermodynamic limiting process. Figure \ref{fig:go_traject}
shows a sample trajectory of the stochastic system, simulated via the NRM.
In~\ref{fig:go_traject}, all parameters are set to nominal values and the only nonzero initial
states are $P_a$ and $P_r$, with one molecule of each. We plot the activator
protein $A$, the repressor protein $R$, and the complex $C$ up to final time $T
= 50$.
Returning to the original question illustrated in Figure~\ref{fig:diagram}, we
will use the sensitivity information gained from the cheaper, deterministic
model (RRE) to make conclusions about parameter importance in the more
expensive, stochastic model (RTC). 

\begin{figure}
\centering
\includegraphics[trim = 5 5 25 30, clip, width=.78\textwidth]{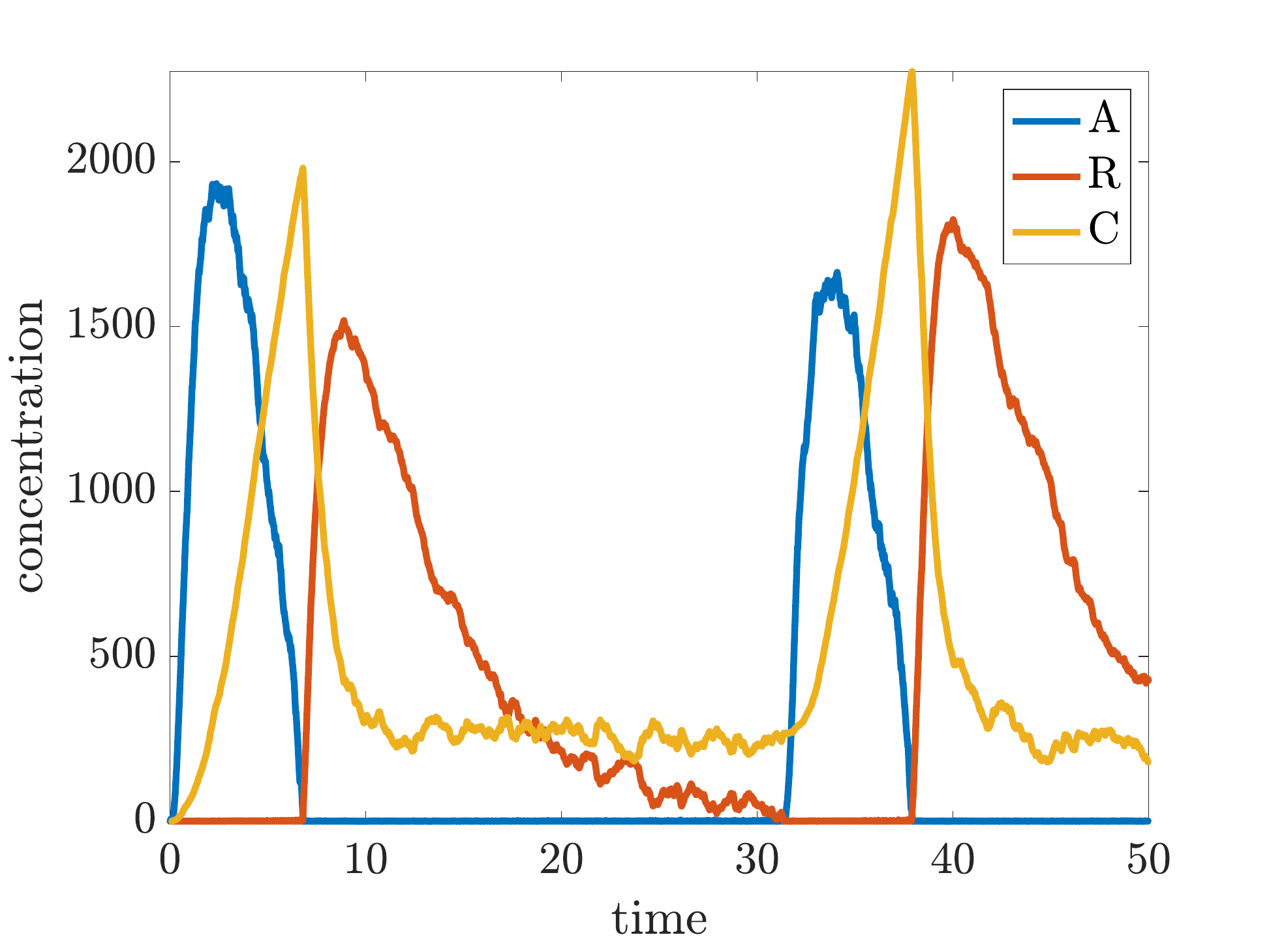}
\caption{Trajectories of the three dominant species at nominal parameters via the NRM.} \label{fig:go_traject}
\end{figure}

We define the stochastic and deterministic QoIs, respectively, as
\[
f_V(\boldsymbol\theta, \omega) = \frac{1}{T}\int_0^T R^V(t; \boldsymbol\theta, \omega)~dt \quad\text{and}\quad f(\boldsymbol\theta) = \frac{1}{T}\int_0^T R(t; \boldsymbol\theta)~dt,
\]
where $R^V$ is the concentration of the repressor computed via the NRM and $R$
is the concentration of the repressor computed as the solution to the
accompanying RRE.  Using the Monte-Carlo method presented
in~\cite{saltelli2010,saltellibook}, we then estimate the total
Sobol' indices for the deterministic model. Figure~\ref{fig:sobol_genetic}
shows the total Sobol' indices.
It is clear that $\alpha_A, \beta_A, \delta_{MA}, $ and
$\alpha_a$ are the four most important parameters, capturing over 50\% of the
variance of the deterministic QoI. 
\begin{figure}[ht]
\centering
\includegraphics[trim = 35 20 25 24,scale=.45]{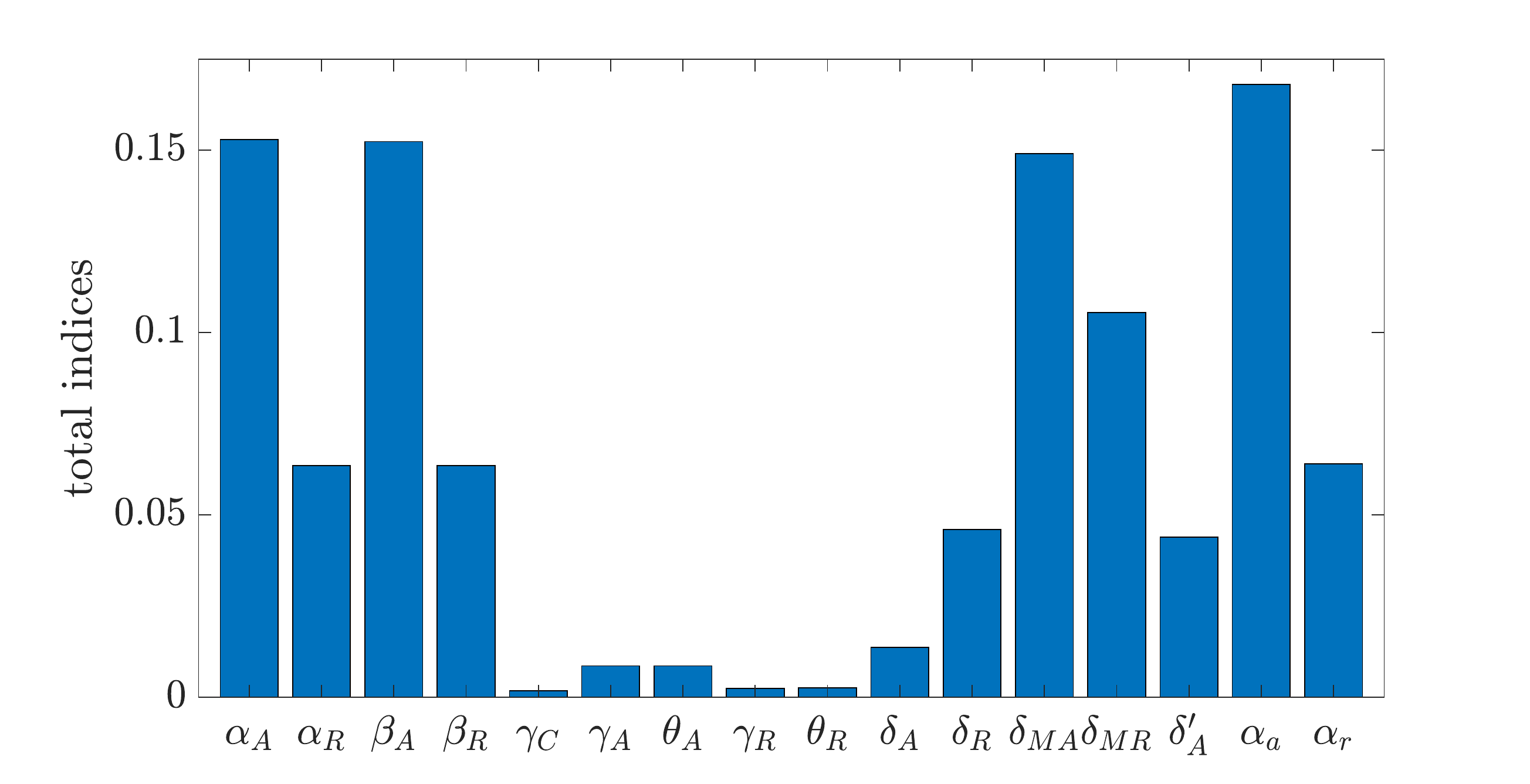}
\caption{Estimated total Sobol' indices for the genetic oscillator RRE.} \label{fig:sobol_genetic}
\end{figure}

We can determine unimportant inputs by putting an importance threshold on the
total Sobol' indices; parameters whose Sobol' index falls below the threshold
will be considered unimportant.  For instance, using $0.02$ as a threshold, we
identify $\gamma_C, \gamma_A, \theta_A, \gamma_R, \theta_R$, and $\delta_A$ as
the six least important parameters, capturing less than 5\% of the variance of
the deterministic QoI.  We then propose a reduced-dimensional model, where the six least important parameters are fixed at their nominal values, reducing the
dimensionality from sixteen to ten. To verify that this lower-dimensional model
remains an accurate representation of the full model, we sample the stochastic
QoI and plot its PDF while fixing and varying the unimportant parameters; see
Figure~\ref{fig:pdfs_fixedparams}.  The red dashed line, corresponding to the
reduced model with the six least important parameters fixed has a negligible
difference with the PDF of the full model.  Increasing the threshold from
$0.02$ to $0.05$ adds $\delta_R$ and $\delta_A^\prime$ to the unimportant
category.  However, as seen in Figure~\ref{fig:pdfs_fixedparams}, the
PDF of the resulting reduced model (dashed green line), obtained by fixing now
eight parameters shows a notable difference with the PDF of the full model.
This illustrates the balance one must strike between fixing unimportant
parameters to reduce parameter dimension and the loss of information that may
result from using a cheaper model.  Finally, we illustrate the impact of fixing the
four most important parameters (black dashed line in
Figure~\ref{fig:pdfs_fixedparams}). This approach fixes every parameter with a total Sobol' index greater than 0.15 ($\alpha_A, \beta_A, \delta_{MA}$, and $\alpha_a$). This results in a substantial underestimation of the variance and a potential loss of valuable model information.

\begin{figure}[htb]
\centering
\includegraphics[trim = 25 0 25 5, clip,width=.75\textwidth]{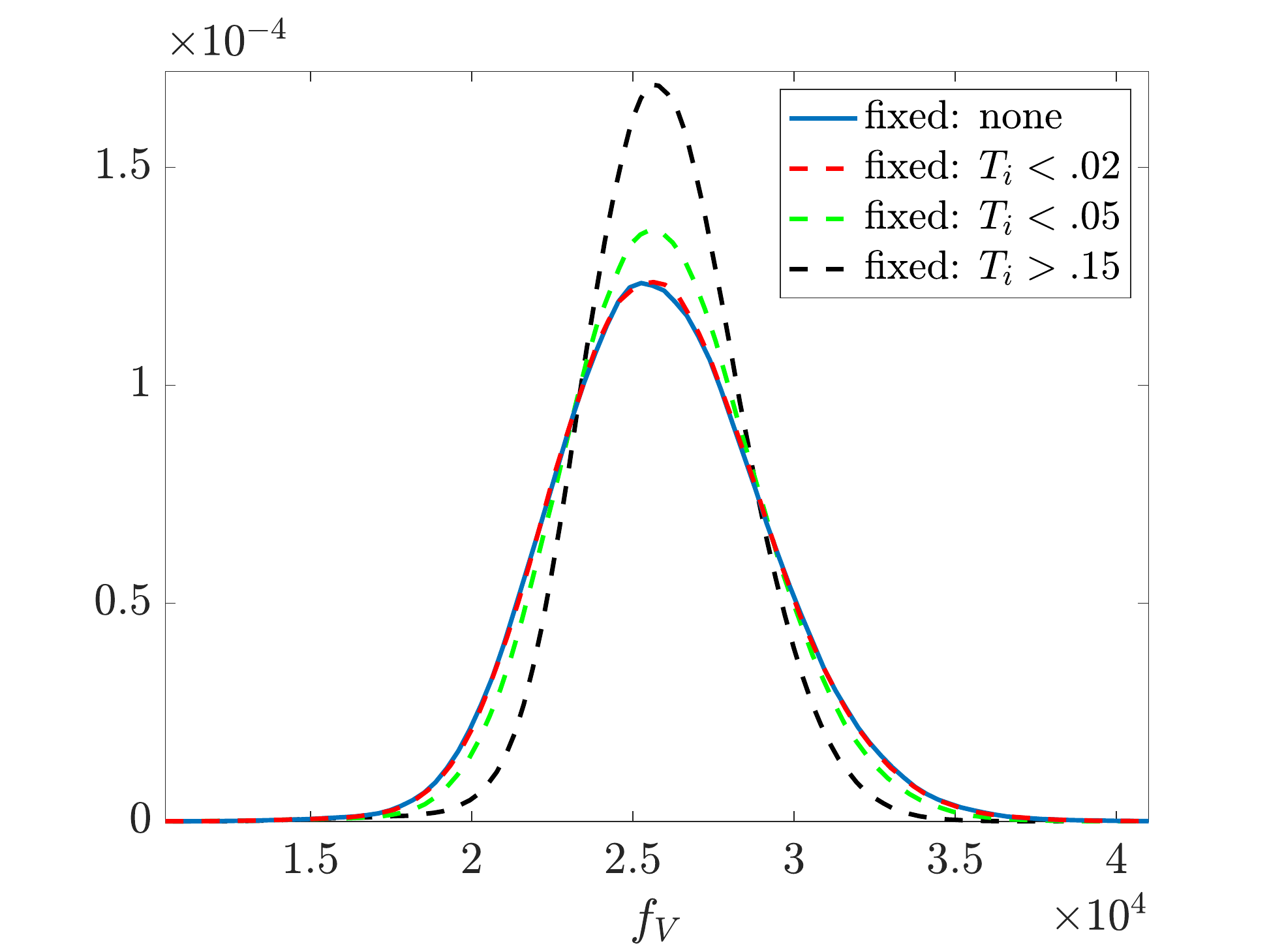}
\caption{PDFs of the stochastic QoI, $f_V$, sampled while fixing the following parameters:  black line ($\alpha_A, \beta_A, \delta_{MA}, \alpha_a$), green line ($\gamma_C, \gamma_A, \theta_A, \gamma_R, \theta_R, \delta_A, \delta_R, \delta_A^\prime$), red line ($\gamma_C, \gamma_A, \theta_A, \gamma_R, \theta_R, \delta_A$), black line without fixed parameters. Total index thresholds are provided for each PDF. } \label{fig:pdfs_fixedparams}
\end{figure}
\pagebreak
\section{Conclusions}
Sensitivity analysis is often performed on simplified surrogate models with the
{\em hope} that (\ref{general-q}) holds; i.e., the hope, explicit or not,  that
the results from the analysis of a surrogate model will hold for the 
full model. We have presented
here a partial result in that direction showing this assertion to be true for
a specific specifc class of problems (chemical systems), 
a specific type of surrogate (obtained from the thermodynamic limit) and a
specific GSA approach (Sobol' indices). Our study not only shows and justifies,
in an arguably restricted framework, that GSA can sometimes be done ``on the
cheap'', we argue that it reflects important properties of the GSA methods
themselves. Further study should consider other types of limiting processes
linking surrogates and full models such as homogenization of differential equations, 
discretization and projections, as well as more general types of GSA methods. 

\section*{Acknowledgements}
We acknowledge the computing resources provided on Henry2, a high-performance
computing cluster operated by North Carolina State University.  We also thank
Andrew Peterson for his assistance with distributed and parallel computations, which was provided through the Office of
Information Technology HPC services at NC State.

\bibliographystyle{siamplain}
\bibliography{mag}

\begin{thebibliography}{10}

\bibitem{ahlw}
{\sc D.~Anderson, D.~Higham, S.~Leite, and R.~Williams}, {\em On constrained
  langevin equations and (bio)chemical reaction networks}, Multiscale Model.
  Simul., 17 (2019).

\bibitem{anderson}
{\sc D.~Anderson and T.~Kurtz}, {\em Continuous time {Markov} chain models for
  chemical reaction networks}, in Design and analysis of biomolecular circuits:
  engineering approaches to systems and synthetic biology, H.~Koeppl, G.~Setti,
  M.~di~Bernardo, and D.~Densmore, eds., Springer, 2011, pp.~3--42.

\bibitem{anderson2007modified}
{\sc D.~F. Anderson}, {\em A modified next reaction method for simulating
  chemical systems with time dependent propensities and delays}, The Journal of
  chemical physics, 127 (2007), p.~214107.

\bibitem{AndersonHigham12}
{\sc D.~F. Anderson and D.~J. Higham}, {\em Multilevel monte carlo for
  continuous time markov chains, with applications in biochemical kinetics},
  Multiscale Modeling \& Simulation, 10 (2012), pp.~146--179.

\bibitem{durrett}
{\sc R.~Durrett}, {\em Probability theory and examples}, Cambridge University
  Press, 2019.

\bibitem{kurtz}
{\sc S.~Ethier and T.~Kurtz}, {\em Markov processes: characterization and
  convergence}, Wiley, 1986.

\bibitem{gibson00}
{\sc M.~Gibson and J.~Bruck}, {\em Efficient exact stochastic simulation of
  chemical systems with many species and many channels}, J. Phys. Chem. A, 104
  (2000), pp.~1876--1889.

\bibitem{gillespie07}
{\sc D.~Gillespie}, {\em Stochastic simulation of chemical kinetics}, Annu.
  Rev. Phys. Chem., 58 (2007), pp.~35--55.

\bibitem{gillespie76}
{\sc D.~T. Gillespie}, {\em A general method for numerically simulating the
  stochastic time evolution of coupled chemical reactions}, J. Comput. Phys.,
  22 (1976), pp.~403--434.

\bibitem{hag}
{\sc J.~Hart, A.~Alexanderian, and P.~Gremaud}, {\em Efficient computation of
  {Sobol}' indices for stochastic models}, SIAM J. Sci. Comput., 39 (2017),
  pp.~A1514--A1530.

\bibitem{higham}
{\sc D.~Higham}, {\em Modeling and simulating chemical reactions}, SIAM Review,
  50 (2008), pp.~347--368.

\bibitem{iooss}
{\sc B.~{Iooss} and P.~{Lema\^itre}}, {\em A review on global analysis
  methods}, in Uncertainty management in simulation-optimization of complex
  systems, G.~{Dellino} and C.~{Meloni}, eds., Springer, 2015, ch.~5,
  pp.~543--501.

\bibitem{janon14}
{\sc A.~{Janon}, M.~{Nodet}, and C.~{Prieur}}, {\em Uncertainties assessment in
  global sensitivity indices estimation from metamodels}, Int. J. Uncert.
  Quant., 4 (2104), pp.~21--36.

\bibitem{surrogate-koziel}
{\sc S.~Koziel, D.~Ciaurri, and L.~Leifsson}, {\em Surrogate based methods}, in
  Computational Optimization, Methods and Algorithms, S.~Koziel and X.~Yang,
  eds., vol.~356 of Studies in Computational Intelligence, Springer, 2011,
  pp.~33--59.

\bibitem{le2015variance}
{\sc O.~Le~Ma{\^\i}tre, O.~Knio, and A.~Moraes}, {\em Variance decomposition in
  stochastic simulators}, The Journal of chemical physics, 142 (2015),
  p.~06B620\_1.

\bibitem{navarro2016global}
{\sc M.~Navarro~Jimenez, O.~Le~Ma{\^\i}tre, and O.~Knio}, {\em Global
  sensitivity analysis in stochastic simulators of uncertain reaction
  networks}, The Journal of chemical physics, 145 (2016), p.~244106.

\bibitem{qian}
{\sc E.~Qian, B.~Peherstorfer, D.~O'Malley, V.~Vesselinov, and K.~Willcox},
  {\em Multifidelity monte carlo estimations of variance and sensitivity
  indices}, SIAM/ASA Uncertainty Quantification, 6 (2018), pp.~683--706.

\bibitem{rathinam2015moment}
{\sc M.~Rathinam}, {\em Moment growth bounds on continuous time markov
  processes on non-negative integer lattices}, Quart. Appl. Math. 73 (2015),
  347-364,  (2015).

\bibitem{rathinam2010efficient}
{\sc M.~Rathinam, P.~W. Sheppard, and M.~Khammash}, {\em Efficient computation
  of parameter sensitivities of discrete stochastic chemical reaction
  networks}, The Journal of chemical physics, 132 (2010), p.~034103.

\bibitem{saltelli2010}
{\sc A.~Saltelli, P.~Annoni, I.~Azzini, F.~Campolongo, M.~Ratto, and
  S.~Tarantola}, {\em Variance based sensitivity analysis of model output.
  design and estimator for the total sensitivity index}, Computer Physics
  Communications, 181 (2010), pp.~259--270.

\bibitem{saltellibook}
{\sc A.~{Saltelli}, M.~{Ratto}, T.~{Andres}, F.~{Campolongo}, J.~{Cariboni},
  D.~{Gatelli}, M.~{Saisana}, and S.~{Tarantola}}, {\em Global sensitivity
  analysis: the primer}, Wiley, 2008.

\bibitem{sheppard2012pathwise}
{\sc P.~W. Sheppard, M.~Rathinam, and M.~Khammash}, {\em A pathwise derivative
  approach to the computation of parameter sensitivities in discrete stochastic
  chemical systems}, The Journal of chemical physics, 136 (2012), p.~034115.

\bibitem{ralph}
{\sc R.~{Smith}}, {\em Uncertainty quantification, theory, implementation, and
  applications}, SIAM, 2013.

\bibitem{sobol93}
{\sc I.~{Sobol'}}, {\em Sensitivity estimates for non linear mathematical
  models}, Math. Mod. Comp. Exp., 1 (1993), pp.~407--414.

\bibitem{sobol}
{\sc I.~Sobol'}, {\em Global sensitivity indices for nonlinear mathematical
  models and their {Monte Carlo} estimates}, Mathematics and Computers in
  Simulation, 55 (2001), pp.~271--280.

\bibitem{vilar2002}
{\sc J.~M. Vilar, H.~Y. Kueh, N.~Barkai, and S.~Leibler}, {\em Mechanisms of
  noise-resistance in genetic oscillators}, Proceedings of the National Academy
  of Sciences, 99 (2002), pp.~5988--5992.

\bibitem{ting}
{\sc T.~Wang}, {\em Parametric sensitivity analysis of stochastic reaction
  networks}, PhD thesis, University of Maryland, Baltimore County, 2015.

\bibitem{WangRathinam16}
{\sc T.~Wang and M.~Rathinam}, {\em Efficiency of the girsanov transformation
  approach for parametric sensitivity analysis of stochastic chemical
  kinetics}, SIAM/ASA Journal on Uncertainty Quantification, 4 (2016),
  pp.~1288--1322.

\bibitem{wilkinson}
{\sc D.~Wilkinson}, {\em Stochastic modelling for systems biology}, CRC Press,
  2~ed., 2012.

\end{thebibliography}

\end{document}